\documentclass[12pt,reqno
 ]{amsart}

\usepackage[
    margin=1in
]{geometry}

\makeatletter
\def\myparagraph{\@startsection{paragraph}{4}{\z@}%
    {1.25ex \@plus1ex \@minus.2ex}%
    {-1em}%
    {\normalfont\normalsize\bfseries}}
\makeatother

\RequirePackage{mathtools}
\RequirePackage{amsopn}
\RequirePackage{amsfonts}
\RequirePackage{paralist}
\RequirePackage{amssymb}
\RequirePackage{amsthm}
\RequirePackage{mathrsfs}
\usepackage{enumitem}

\usepackage{url,tensor}
\usepackage{stix}
\usepackage{tikz-cd}
\usepackage{stmaryrd} 
\tikzset{
	Subseteq/.style={
		draw=none,
		every to/.append style={
			edge node={node [sloped, allow upside down, auto=false]{$\subset$}}}
	},
	Supseteq/.style={
	draw=none,
	every to/.append style={
		edge node={node [sloped, allow upside down, auto=false]{$\supseteq$}}}
		}
}

\usepackage{enumitem}
\usepackage{hyperref}

\newtheorem{thm}{Theorem}[section] 
\newtheorem{thmx}{Theorem}
    
\numberwithin{equation}{section}

\newtheorem{cor}[thm]{Corollary}
\newtheorem{corollary}[thm]{Corollary}
\newtheorem{lemma}[thm]{Lemma}
\newtheorem{prop}[thm]{Proposition}

\theoremstyle{definition}
\newtheorem{defn}[thm]{Definition}

\newtheorem*{claim*}{Claim}
\newtheorem{assumption}[thm]{Assumption}
\newtheorem{notation}[thm]{Notation}
\newtheorem*{notationconvention}{Notation and conventions}

\theoremstyle{remark}
\newtheorem{remark}[thm]{Remark}
\newtheorem{hooptedoodle}[thm]{Hooptedoodle}
\newtheorem{example}[thm]{Example}

\newtheorem{problem}[thm]{Problem}

\usepackage[backend=biber,
	giveninits=true,
	hyperref=true, 
    minalphanames=6,
    maxalphanames=10,
	block=none,
    style=alphabetic,
    url=false, 
	isbn=false]{biblatex}	
\newcommand{\etale}{{\'e}tale}
\newcommand{\LCH}{locally compact, Hausdorff}

\newcommand{\E}{\mathcal{E}}  
\newcommand{\cK}{\mathcal{C}} 
\newcommand{\cM}{\mathcal{M}} 
\newcommand{\cN}{\mathcal{N}} 
\newcommand{\cP}{\mathcal{P}} 
\newcommand{\cQ}{\mathcal{Q}} 
\newcommand{\cU}{\mathcal{U}} 
\newcommand{\mfs}{\mathfrak{s}} 
\newcommand{\mft}{\mathfrak{t}} 

\newcommand{\z}{^{(0)}}
\newcommand\comp{^{(2)}}
\newcommand\inv{^{-1}}
\newcommand{\Iso}[1]{\operatorname{Iso}\left(#1\right)} 
\newcommand{\Int}[2][]{\operatorname{int}_{#1}\left(#2\right)} 

\DeclareMathOperator{\supp}{supp}
    \newcommand\suppo{\supp^{\circ}}
    \newcommand\suppoG{\supp_{G}^{\circ}}
    
\DeclareMathOperator{\dom}{dom}

\newcommand\norm[1]{\|#1\|}

\def\mathcs{C^{*}}
\newcommand{\cs}{\ensuremath{\mathcs}}
\newcommand{\red}{r}
\newcommand{\csr}{\cs_{\red}}

\let\ipscriptstyle=\scriptscriptstyle
\def\lipsqueeze{{\mskip -3.0mu}}
\def\ripsqueeze{{\mskip -3.0mu}}
    \newcommand{\bfp}[2]{\lipsqueeze\tensor*[_{\ipscriptstyle #1}]{\times}{_{\ipscriptstyle #2}}\ripsqueeze} 
    \newcommand{\bfpsr}{\bfp{s}{r}}

\newcommand{\intA}[1]{{#1}^{\sharp}} 

\newcommand{\twPD}[1][\E]{\widehat{S}^{#1}}

\newcommand{\actiongpd}[3][]{{#3 \tensor*[_{#1}]{\ltimes}{} #2}} 
    \newcommand{\Weylgpd}[1][\twPD]{\actiongpd{#1}{(G/S)}}
    \newcommand{\Weyltwist}[1][\twPD]{\actiongpd{#1}{(\E/S)}}

\newcommand{\normalgeq}{\mathrel{\unrhd}} 
\newcommand{\atgext}{\E\to G \normalgeq S} 

\newcommand{\cocycle}{\operatorname{\mathsf{c}}}
\newcommand{\coboundary}{\operatorname{\mathsf{d}}}

\renewcommand{\vec}[1]{\mathbf{#1}}

    \newcommand\mvisiblespace[1][.7em]{%
        	\makebox[#1]{%
        		\kern.07em
        		\vrule height.3ex
        		\hrulefill
        		\vrule height.3ex
        		\kern.07em
        	}
        }

\newcommand{\HleftX}{
        \mathchoice
              {\displaystyle\mathbin\smalltriangleright}
              {\textstyle\mathbin\smalltriangleright}
              {\scriptstyle\mathbin\smalltriangleright}
              {\scriptscriptstyle\mathbin\smalltriangleright}
        }

\newcommand{\XleftG}{
        \mathchoice
              {\displaystyle\mathbin\smallblacktriangleright}
              {\textstyle\mathbin\smallblacktriangleright}
              {\scriptstyle\mathbin\smallblacktriangleright}
              {\scriptscriptstyle\mathbin\smallblacktriangleright}
        }

\begin{document}

\begin{abstract}
We identify which conditions on an open normal subgroupoid $S$ of a \LCH, \etale\ groupoid $G$ with twist $\E$ are necessary and sufficient for the subgroupoid's reduced twisted $\cs$-algebra $\csr(S;\E_{S})$ to be a $\cs$-diagonal in the ambient groupoid $\cs$-algebra $\csr(G;\E)$. We do so by first giving an explicit description of the Weyl groupoid and Weyl twist associated to any non-traditional Cartan subalgebra, that is, a Cartan subalgebra that is induced from a non-trivial open normal subgroupoid, as studied in \cite{DWZ:Twist}. We then combine this description with Kumjian-Renault theory to establish the necessary and sufficient conditions to get a $\cs$-diagonal.
\end{abstract}

\thanks{I was supported by an FWO Senior Postdoctoral Fellowship (1206124N) during my time at KU Leuven. I would like to thank Carlos Campoy, Elizabeth Gillaspy, Bartosz Kwaśniewski, Aidan Sims, and Stefaan Vaes for helpful discussions. I am further indebted to Jean Renault for helping me better understand the relationship of the current work to that in \cite{IKRSW:2021:Extensions,IKRSW:2021:Pushout,Renault:2023:Ext}}
\author[A. Duwenig]{Anna Duwenig}
\address{School of Mathematics and Statistics, UNSW Sydney, Australia}
\email{a.duwenig@unsw.edu.au}
\title{Non-traditional $\cs$-diagonals in twisted groupoid $\cs$-algebras}

\maketitle

\section{Introduction}

This paper is concerned with open normal subgroupoids~$S$ of \etale\ groupoids~$G$, and the inclusions of twisted $\cs$-algebras $\cs_{r}(S;\E_{S})\subset\cs_{r}(G;\E) $ that they induce. Our focus is on the situation in which the subalgebra is Cartan, and we will determine necessary and sufficient conditions on~$S$ for this subalgebra to even be a $\cs$-diagonal. To motivate this endeavor, let us start by discussing the case where $S=G\z$ is the groupoid's unit space, which is the `traditional' situation that was solved by Kumjian and Renault.

Since~$G$ is \etale, $G\z$ is an open subset, and so the map ``extension by zero'' 
   extends from $C_{c}(G\z )\to C_{c}(G;\E )$ to an inclusion
  $C_0(G\z )\subset \cs_{r}(G;\E)$. 
  Renault showed in~\cite[Theorem 5.2]{Renault:Cartan}  that this subalgebra is maximal abelian in $\cs_{r}(G;\E)$ if and only if~$G$ is effective, i.e., the interior of the isotropy is nothing more than the unit space $G\z $. Work of Kumjian \cite{Kum:Diags}, predating that of Renault, found that $C_0(G\z )$ furthermore has the {\em unique extension property} if and only if $G$ is even principal, meaning that the unit space is the entire isotropy rather than just its interior. In summary, certain topological and algebraic properties of the groupoid $G$ are reflected in the inclusion of $C_0(G\z )$ in $ \cs_{r}(G;\E)$.
  
  What is most remarkable about Kumjian and Renault's work, however, is that it provides a recipe for how to reconstruct the groupoid $G$ and twist $\E$ from the masa inclusion $C_0(G\z )\subset \cs_{r}(G;\E)$. For this, it was necessary to realize not only that (1) $C_0(G\z )$ is masa, but that it satisfies two additional properties: (2) there exists a faithful conditional expectation $\cs_{r}(G;\E)\to C_0(G\z )$, induced by the map that restricts $C_c$-functions on $\E$ to $G\z$; and (3) $C_0(G\z )$ is regular in $\cs_{r}(G;\E)$, meaning that its normalizers generate the ambient groupoid $\cs$-algebra. These three properties combined make $C_0(G\z )\subset \cs_{r}(G;\E)$ a {\em Cartan inclusion}, and  Kumjian--Renault theory allows you to start with such a Cartan inclusion $B\subset A$ and `reconstruct' the twisted groupoid $\E\to G$ that realizes $A$ as its $\cs$-algebra and $B$ as the traditional Cartan subalgebra, $C_0(G\z)$. This paper will use this reconstruction recipe, but applied to {\em non-traditional} Cartan subalgebras:

 When $G$ is not effective, meaning that it contains an open abelian subgroupoid~$S$ that is strictly larger than $G\z $ \cite[Lemma 3.10]{DWZ:2025:Twist}, then $C_0(G\z )$ is  properly contained in the abelian $\csr(S;\E_{S})$ and is hence not a masa in $\csr(G;\E)$. 
  But there is a chance that the subalgebra $\csr(S;\E_{S})$ is as well behaved as $C_0(G\z)$:  If $S$ is closed, we get a conditional expectation $\csr(G;\E)\to\csr(S;\E_{S})$, likewise induced by restriction of functions \cite[Lemma 3.4]{BEFPR:2021:Intermediate}; and if~$S$ is normal, then $\csr(S;\E_{S})$ is regular in $\csr(G;\E)$ \cite[Corollary 3.25]{DWZ:2025:Twist}. In other words, any clopen, normal subgroupoid $S$ of $G$ gives rise to a subalgebra that satisfies properties (2) and (3) above. But when does it satisfy (1), i.e., when is it masa? The answer was given in \cite{DWZ:2025:Twist}, building on previous work in \cite{DGN:Weyl}, and can be summarized as follows, where we write for $u\in G\z$ and $e\in\E(u)$,
    \[
        [e, \E_{S} ]
        \coloneqq  
        \left\{
        e\inv\sigma\inv e\sigma
        :
        \sigma\in  \E_{S} (u)
        \right\}
     \]
     and where $\Int[\E]{X}$ denotes the interior in $\E$ of a subset $X$.
\begin{thm}[{\cite[Corollary 3.23]{DWZ:2025:Twist}}]\label{thm:DWZ}
   Suppose that $\E$ is a twist over a \LCH, \etale\ groupoid~$G$, and that~$S$ is an open and normal subgroupoid of~$G$. The following statements  are equivalent.
  \begin{enumerate}[label=\textup{(\roman*)}]
    
    \item\label{it:thm:DWZ:B} $\csr(S;\E_{S})$ is a Cartan subalgebra of $\csr(G;\E)$. 
    
    \item\label{it:thm:DWZ:S} $S$ is closed in~$G$ and satisfies
      \begin{align} 
      G\z \subset \E_{S}
      &=
      \Int[\E]{
      \left\{
      e\in \Iso{\E}: 1=\left|[e,\E_{S}]\right|
      \right\}
      },
      \label{eq:thm:DWZ:maximal}
      \tag{\texttt{max}}
      \\
        \label{eq:thm:DWZ:ricc}
      \emptyset
      &=
      \Int[\E]{
      \left\{
      e\in \Iso{\E}: 1<\left|[e,\E_{S}]\right|=\left|[\pi(e),S\vphantom{\E_{S}}
      ]\right|<\infty
      \right\}
      }     
      \tag{\texttt{ricc}}
        .
    \end{align}
  \end{enumerate}
\end{thm}
Note that~\eqref{eq:thm:DWZ:maximal} says that $\E_{S}$ is maximal among open abelian subgroupoids of $\E$, and that this condition is not sufficient to force maximality of the subalgebra. Rather, one also needs to assume~\eqref{eq:thm:DWZ:ricc}; see \cite[Remark 3.6]{DWZ:Twist} for the idea behind this condition.
In light of the ``traditional'' Cartan subalgebra  \'a la Renault, i.e., $C_0(G\z)$ for an effective groupoid~$G$, we will refer to a Cartan subalgebra of the form $\csr(S;\E_{S})$ as a {\em non-traditional Cartan subalgebra}, and we will often write $\atgext$ for the original input.

As alluded to earlier, Renault and Kumjian showed a surprising converse: From any Cartan pair $B\subset A$ of  $\cs$-algebras, one can construct an effective \etale\ groupoid $W\coloneqq  W_{B\subset A}$ with twist $\Sigma\coloneqq  \Sigma_{B\subset A}$ (called {\em the Weyl groupoid} and {\em the Weyl twist}, respectively) such that there exists a Cartan-preserving isomorphism $A\cong \csr( W ; \Sigma)$, and $B$ is a $\cs$-diagonal if and only if $W$ is principal. The purpose of this paper, then, is two-fold:
\begin{enumerate}
    \item We identify the Weyl pair associated to a non-traditional Cartan subalgebra as in Theorem~\ref{thm:DWZ} entirely in terms of the original data $\atgext$ (Theorem~\ref{thm:A}); this generalizes the main results in \cite{DGN:Weyl}.
    \item We refine the equivalent conditions of \cite[Theorem 3.3]{DWZ:2025:Twist} to $\cs$-diagonals (Theorem~\ref{thm:B}). To the author's knowledge, this has not been attempted before.
\end{enumerate}

The paper is structured as follows. We recall necessary definitions and facts in Section~\ref{sec:prelim}: we first review twisted groupoid C*-algebras in Subsection~\ref{ssec:twisted algebras} and Cartan subalgebras and the associated Weyl twist and groupoid in Subsection~\ref{ssec:Cartan}. In Subsection~\ref{ssec:actions}, we remind the reader that the input $\atgext$ implies that the groupoid $\E$ acts continuously on the $\E$-twisted Pontryagin dual $\twPD$ of~$S$ (Definition~\ref{def:twPD}).

In Section~\ref{sec:Weyl}, where we restrict ourselves to \etale\ groupoids, we give an explicit homeomorphism 
between the twisted dual $\twPD$ of~$S$ and the Gelfand dual of the commutative $\cs$-algebra $\csr(S;\E_{S})$ (Proposition~\ref{prop:spec B}). We can then state and prove our first theorems: assuming that~$G$ is \etale\ and that $\csr(S; \E_{S} )$ is a  Cartan subalgebra of $\csr(G;\E)$ (i.e, $S$ is as in Theorem~\ref{thm:DWZ}\ref{it:thm:DWZ:S}), we give the following explicit description of the Weyl twist $\Sigma_{B\subset A}$ and the Weyl groupoid $W_{B\subset A}$ of this non-traditional Cartan pair:
\begin{thmx}\label{thm:A}

    Suppose that $\E$ is a twist over a \LCH, \etale\ groupoid~$G$, and that~$S$ is an open and normal subgroupoid of~$G$ such that
    $\csr(S; \E_{S} )$ is a Cartan subalgebra of $\csr(G;\E)$. Then the associated Weyl groupoid and twist are canonically isomorphic to the action groupoid
    $\Weylgpd $ respectively to the quotient of $\actiongpd{\twPD}{\E}$ by the relation:
    \quad
    $(e,\kappa)\sim (f,\kappa) \;:\Leftrightarrow\; f\inv e\in \E_{S}$ and $\kappa(f\inv e)=1$.
\end{thmx}
The exact statements, including the explicit groupoid isomorphisms, can be found in Theorems~\ref{thm:Weyl groupoid} and~\ref{thm:Weyl twist}. Subsection~\ref{ssec:corollaries:A} contains some corollaries of Theorem~\ref{thm:A}. For example,
if~$G$ is second countable, the theorem   implies that the action of $G/S$ on $\twPD $ is topologically free (Corollary~\ref{cor:effective and top princ}). We further recover some special cases of results found in \cite{IKRSW:2021:Pushout} and \cite{Renault:2023:Ext}, namely that $\csr(G;\E)$ is canonically isomorphic to 
\[
\csr\left( \Weylgpd  ; (\actiongpd{\twPD}{\E})/\sim \right).
\]

In Section~\ref{sec:trivial Weyl twist}, we discuss which conditions on the input $\atgext $  are necessary and sufficient for the Weyl twist $\Sigma_{B\subset A}$ to be trivializable.

Section~\ref{sec:diag} is devoted to refining the equivalent conditions of 
\cite[Theorem 3.3]{DWZ:2025:Twist}
from Cartan subalgebras to $\cs$-diagonals. To be precise, we use our description of the Weyl groupoid to ascertain 
when 
the subalgebra $\csr(S; \E_{S} )$ is a $\cs$-diagonal in $\csr(G;\E)$. A consequence of our main result, Theorem~\ref{thm:diag}, is:
\begin{thmx}\label{thm:B}
    Suppose 
  $\E$ is a twist 
  over a \LCH, \etale\
  groupoid~$G$, and~$S$ is an open and normal subgroupoid of~$G$. Then the following  are equivalent.
  \begin{enumerate}[label=\textup{(\roman*)}]
    
    \item 
    $\csr(S; \E_{S} )$ is a $\cs$-diagonal in $\csr(G;\E)$.
    
    \item
      The restricted twist
      $ \E_{S} $ is abelian and closed in $\E$, and for every $e\in \E(u)\setminus \E_{S}(u)$,  there exists $\sigma\in \E_{S}(u) $ and $z\in\mathbb{T}\setminus\{1\}$ with $e \sigma= z\cdot \sigma e$.
  \end{enumerate}
\end{thmx}

Subsection~\ref{ssec:corollaries:B} contains corollaries and applications of Theorem~\ref{thm:B}. A surprising one is Corollary~\ref{cor:diag:untwisted}: if the original datum is untwisted (i.e., $\E=\mathbb{T}\times G$ is trivial), then the only possible choice for~$S$ is $\Iso{G}$; no other normal subgroupoid can give rise to a $\cs$-diagonal in $\csr(G)$. We end in Section~\ref{sec:questions} with some open problems.

\section{Preliminaries}\label{sec:prelim}

\begin{notationconvention}
    All groupoids are assumed to be topological groupoids that are {\em locally compact and Hausdorff}, so these descriptors will rarely be repeated. (We do not always assume groupoids to be \etale\ or second countable.)
\end{notationconvention}

\begin{defn}\label{def:gpd adjectives}
    For a unit $u\in G\z$ of a groupoid~$G$, we write
$uG\coloneqq r\inv (\{u\})$ and
$Gu \coloneqq s\inv (\{u\})$. The \emph{isotropy group} at $u$ is
denoted $G(u)\coloneqq uG\cap Gu $, and the \emph{isotropy
  bundle} is $\Iso{G}=\bigcup_{u\in G\z}G(u)$. We say that~$G$ is
    \begin{itemize}    
        \item \emph{abelian} if it is a bundle of abelian
groups;         
        \item {\em principal} if $\Iso{G}=G\z$;
        \item {\em topologically principal} if $\{u\in G\z : G(u)=\{u\}\}$ is dense in $G\z$; 
        \item {\em effective} if $\Int[G]{\Iso{G}}=G\z$; 
\item 
\emph{\etale} if the range and source maps are local homeomorphisms, and any $U\subset G$ for which $r|_{U}\colon U\to r(U)$ is a homeomorphism, is called a {\em bisection}.
    \end{itemize}
     
  We say that a subgroupoid~$S$ of~$G$ is \emph{normal} if $S\subset \Iso{G}$ and if $Sg=g S$ for all $g\in G$.
\end{defn}

\begin{remark}
    \begin{enumerate}
        \item 
For (\LCH) \etale\ groupoids, topologically principal implies effective, and the two notions coincide in the case that the groupoid is second countable \cite[Proposition 3.6]{Renault:Cartan}.
    \item It is not difficult to see that~$G$ is principal if and only if it has no non-trivial abelian subgroupoids, and it is effective if and only if it has no non-trivial {\em open} abelian subgroupoids
    \cite[Lemma 3.10]{DWZ:Twist}.
    \end{enumerate}
\end{remark}

\begin{defn}[{cf.\ \cite[Def.\ 2.1]{Wil2019}}]\label{def:gpd action}
	For a \LCH\ space $X$, a groupoid~$G$, and a continuous 
    surjection $\rho\colon X\to G\z$, we say that {\em~$G$ acts from the left on $X$ with momentum map $\rho$} if there exists a continuous map 
    \[
    \mvisiblespace\HleftX\mvisiblespace\colon G\bfp{s}{\rho}X\coloneqq \{(g,x)\in G\times X: s(g)=\rho(x)\}\to X
    \]
    such that
    \begin{enumerate}
        \item $\rho(x)\HleftX x = x$ for all $x\in X$,
        \item $\rho(g\HleftX x)=r(g)$, and
        \item $h\HleftX (g\HleftX x)=(hg)\HleftX x$ whenever one (and hence both) sides make sense.
    \end{enumerate}
    In this case, we call $X$ a {\em left $G$-space}.
    We say that $\HleftX$ is
    \begin{itemize}
        \item {\em free} if $g\HleftX x = x$ implies $g=\rho(x)$;
        \item {\em topologically free} if $\{x\in X : g\HleftX x =x \implies g=\rho(x)\}$ is dense in $X$.
    \end{itemize}
    One can make analogous definitions for right actions.
\end{defn}

\begin{defn}\label{def:action groupoid}
    Suppose~$G$ acts from the left on $X$ as in Definition~\ref{def:gpd action}. We equip the space $G\bfp{s}{\rho} X$ with the structure of a groupoid as follows: we define range and source by $r(g,x)=(r(g),g\HleftX x)$ and $s(g,x)=(s(g),x)$, and we give it the structure maps
    \[
        (h,g\HleftX x)(g,x)
        =
        (hg, x)
        \quad\text{ and }\quad
        (g,x)\inv = (g\inv, g\HleftX x).
    \]
    This groupoid is called the {\em action groupoid}, and we will denote it by $\actiongpd{X}{G}$. We furthermore identify its unit space with $X$ via the homeomorphism $(\rho(x),x)\mapsto x$.
\end{defn}

We give a small example of an action groupoid in Example~\ref{ex:rotation as transformation}.

\begin{remark}\label{rmk:free=principal etc}
    It is easy to check that an action is (topologically) free as defined in Definition~\ref{def:gpd action} if and only if the action groupoid is (topologically) principal as defined in Definition~\ref{def:gpd adjectives}.
\end{remark}

\subsection{Twisted groupoid $\cs$-algebras}\label{ssec:twisted algebras}

A {\em groupoid extension of $G$ by $\E$} is a sequence
\begin{equation}
  \label{eq:def twist}
  \begin{tikzcd}
    \mathbb{T}\times G\z\arrow[r,"\iota",hook] &\E \arrow[r,"\pi",two heads]&
    G
  \end{tikzcd}
\end{equation}
where $G$ and $\E$ are (topological) groupoids, and where the surjective $\pi$ and the injective $\iota$ are continuous and open groupoid homomorphisms such that 
\begin{enumerate}
    \item $\iota(\pi(r(e)),1)=r(e)$ for all $e\in \E$,
    \item $\pi$ induces a homeomorphism of the unit space of $\E$ with $G\z$,  and
    \item $\iota$ is a homeomorphism onto $\pi\inv(G\z)$.
\end{enumerate}
We identify the unit space of $\E$ with $G\z$, and we usually denote it by $\cU$. 
The extension is called {\em central} if
\begin{enumerate}[resume]
  \item
  $z\cdot e \coloneq \iota\bigl(z,r(e)\bigr)e=e\iota\bigl (z,s(e) \bigr)$ for
    all $e\in \E$ and $z\in \mathbb{T}$.
\end{enumerate}
As is standard in the $\cs$-algebra literature, we will usually just say that ``{\em $\E$ is a twist  over $G$}'' instead of ``{\em 
\eqref{eq:def twist} is a central groupoid extension}'',
often without explicitly mentioning the maps $\iota$ and $\pi$. As explained in \cite[p.\ 5]{CDGaHV:2024:Nuclear}, it follows that $\E$ is locally trivial (in fact, a principal $\mathbb{T}$-bundle) and that $\pi$ is proper. 

In this paper, we will only be concerned with the reduced twisted groupoid $\cs$-algebra $\csr(G;\E)$ and only in the case  that~$G$ is \etale. For convenience, we remind the reader of its definition:  We define
\begin{equation*}
  \label{eq:1}
  C_{0}(G;\E)=\{f\in C_{0}(\E):\text{$f(z\cdot e)=zf(e)$ for all
      $z\in\mathbb{T}$ and $e\in \E$}\},
\end{equation*}
and $C_{c}(G;\E)=\{f\in C_{0}(G;\E):f\in C_{c}(\E)\}$. We make $C_{c}(G;\E)$ into a $*$-algebra by giving it the convolution product and involution given by
\begin{equation*}
  f_{1}*f_{2}(e') = \sum_{\{\pi(e)\in G:r(e)=r(e')\}} f_{1}(e)f_{2}(e\inv e')
    \quad\text{ and }\quad
  f^{*}(e)=\overline{f(e\inv )}
  ,
\end{equation*}
respectively.
(The sum is to be read as being over the countable set
  $G^{r(e')}$ rather than the uncountable set~$\E^{r(e')}$, and that the formula is well-defined because both $f_{1}$ and $f_{2}$ are $\mathbb{T}$-equivariant and because $\pi(e)=\pi(e_0)$ implies $e_{0}=z\cdot e$ for a unique $z\in\mathbb{T}$.)

For each $u\in  \cU$, let $\ell^{2}(Gu;\E u)$ denote those $\mathbb{T}$-equivariant functions $\xi\colon \E \to \mathbb{C}$ for which
\begin{equation*}
\sum_{\{\pi(e)\in G:s(e)=u\}}|\xi(e)|^{2}<\infty.
\end{equation*}
We define a representation $\theta^{G,\E}_u$ of $C_c(G;\E)$ on $\ell^{2}(Gu;\E u)$ by letting
    \begin{equation*}
      \bigl[\theta^{G,\E}_{u}(f)(\xi)\bigr](e')=\sum_{\{\pi(e)\in G: r(e)=r(e')\}} f(e)\xi(e\inv e')
      \quad\text{for $\xi\in \ell^{2}(Gu;\E u)$ and $e'\in \E u$},
    \end{equation*}
and we let
\[
    \norm{f}_\red \coloneqq \sup_{u\in \cU} \norm{\theta^{G,\E}_u(f)}.
\]
The completion of $C_c( G ;\E)$ in $\norm{\mvisiblespace}_\red$ is, by definition, the \emph{reduced twisted groupoid $\cs$-algebra $\csr ( G ;\E)$}.

\begin{example}
    Of course, the trivial twist $\E=\mathbb{T}\times G$ is a valid choice. In this case, the twisted $\cs$-algebra $\csr(G;\mathbb{T}\times G)$ is canonically isomorphic to the untwisted $\cs$-algebra $\csr(G)$. Note that any {\em trivializable twist}, i.e., one that is isomorphic (as twists) to the trivial twist, also yields a $\cs$-algebra isomorphic to $\csr(G)$.
\end{example}

We remind the reader of a typical source for twists that is more general than the above but still tractable:
\begin{example}
    Any (continuous, $\mathbb{T}$-valued) normalized $2$-cocycle $\cocycle $ of a groupoid~$G$ gives rise to a twist  over~$G$ which we will denote by 
$\mathbb{T}\times_{\cocycle }G$. To be more precise, $\cocycle \colon G\comp\to\mathbb{T}$ is a continuous map that satisfies
\[
    \cocycle (g,h)\cocycle (gh,k)=\cocycle (g,hk)\cocycle (h,k)
    \quad\text{ and }\quad
    \cocycle (g,s(g))=1=\cocycle (r(g),g)
\]
for all $(g,h,k)\in G^{(3)}$. As a topological space, $\mathbb{T}\times_{\cocycle }G$ is just $\mathbb{T}\times G$; since here, we will mainly care about the case where $G=\mathbb{Z}^2$, we will  write elements of $\mathbb{T}\times_{\cocycle }G$ as $(z;g)$ in an effort to avoid double brackets. We equip $\mathbb{T}\times_{\cocycle }G$ with the multiplication and inversion given by 
\[
(z;g)(w;h)=(zw\,\cocycle_{\theta}(g,h);gh)
    \quad\text{ and }\quad
    (z;g)\inv = \left(\overline{z \,\cocycle (g,g\inv)};g\inv\right),
\]
respectively.
To realize $\mathbb{T}\times_{\cocycle }G$ as a twist over~$G$, one defines $\iota(z,u)=(z;u)$ and $\pi(z;g)=g$ for $z\in \mathbb{T}$, $u\in \cU$, $g\in G$. 
\end{example}

It is well known that a continuous section of a twist $\pi\colon\E\to G$ realizes the twist as isomorphic to that coming from a $2$-cocycle (see \cite[Fact 4.1]{Kum:Diags}, \cite[Proposition
    I.1.14]{Renault:gpd-approach}, \cite[Remark 11.1.6]{Sims:gpds}). Moreover, the formula for that $2$-cocycle reveals: if the section is homomorphic, then the $2$-cocycle is constant $1$, i.e., the twist $\E$ is isomorphic to the trivial twist. Let us record this observation here in more detail for ease of reference:

\begin{lemma}\label{lem:c_s}
    Suppose there exists a continuous (set-theoretic) section  $\mfs \colon G\to  \E $ of  $\pi$ which satisfies $\mfs (u)=u$ for all $u\in \cU$.  For $ e \in  \E$, let $z_{ e }\in\mathbb{T}$ be the unique element such that $z_{ e }\cdot \mfs (\pi( e ))= e $. 

    \begin{enumerate}[label=\textup{(\arabic*)}]
        \item\label{it:c_s is 2-cocycle} The map \(
    \cocycle_{\mfs }\colon G^{(2)}\to\mathbb{T}
    \) determined by the equality
    \[
    \iota \bigl( \cocycle_{\mfs }(g,h),u\bigr)
    =
    \mfs (g)\mfs (h)\mfs (g h)\inv
    \text{ for }u=r(g)
    \]
    is a continuous, normalized $2$-cocycle and satisfies
    \begin{equation}\label{eq:c vs z_sigma}
        \cocycle_{\mfs }\bigl(\pi( e ),\pi( f )\bigr) = \overline{z_{ e }} \overline{z_{ f }} z_{ e  f }.
    \end{equation}
    \item\label{it:psi is iso} If $\mathbb{T}\times_{\mfs }G$ denotes the twist over~$G$ induced by $\cocycle_{\mfs }$, then map $\psi\colon  \E\to \mathbb{T}\times_{\mfs }G$ given by
    \[
    \psi( e )= \bigl(z_{ e },\pi( e )\bigr)\]
    is an isomorphism of topological groupoids with inverse $(z,t) \mapsto z \cdot \mfs (t)$. 
    \end{enumerate} 
    In particular, $\E$ is abelian if and only if $\cocycle_{\mfs }$ is symmetric. Moreover, a twist is trivializable if and only if $\pi$ admits a continuous section that is homomorphic.
\end{lemma}

\begin{example}\label{ex:coboundary}
    If a $2$-cocycle $\cocycle$ is a coboundary, meaning that $\cocycle(g,h)=\coboundary(g)\coboundary(h)\coboundary(gh)\inv$ for a continuous map $\coboundary\colon G\to \mathbb{T}$, then one can easily check that the continuous section $\mfs\colon G\to \mathbb{T}\times_{\cocycle}G$ given by $\mfs(g)=(\overline{\coboundary(g)};g)$ is a homomorphism, so that the twist $\mathbb{T}\times_{\cocycle}G$ is trivializable.
\end{example}

Not every twist is induced by a continuous $2$-cocycle, i.e., admits a continuous section as in the above lemma (homomorphic or not); see for example \cite[Theorem A]{ANSZ:2025:twist}. However, if we restrict our attention to the twist $\E(u)\to G(u)$ for any fixed $u\in\cU$ and if $G$ is \etale, then continuity is automatic; in other words, each {\em fibre} of the twist $\E$ is isomorphic to a twist induced by a $2$-cocycle. 

\bigskip

Throughout, we will return to the following, well-studied examples to showcase our results. Both provide models for the rotation algebra, that is, the algebra universally generated by two almost-commuting unitaries:  for some fixed $\theta\in\mathbb{R}$, we let
\[
    A_{\theta}
    \coloneqq  
    \cs 
    \left(U,V : U, V\text{ unitaries such that } VU = \lambda UV
    \right),
        \quad\text{ where }\quad 
        \lambda\coloneqq  \mathsf{e}^{2\pi i \theta}.
\]
The first model is given by a twist on $\mathbb{Z}^2$ induced by a $2$-cocycle, and the second is an action groupoid.
\begin{example}[Rotation algebra as twisted group $\cs$-algebra]\label{ex:c_theta:setup}
    For a fixed $\theta\in\mathbb{R}$, we define a $2$-cocycle on $G=\mathbb{Z}^{2}$ by
    \begin{equation}\label{eq:c_theta}
        \cocycle_{\theta}(\vec{g},\vec{h})=\lambda^{g_{2}h_{1}}
        \quad\text{ where }\quad 
        \lambda\coloneqq  \mathsf{e}^{2\pi i \theta}
        .
    \end{equation}
    The resulting twisted group $\cs$-algebra $\csr(G,\cocycle_{\theta})=\csr(G;\mathbb{T}\times_{\cocycle_{\theta}}G)$ is canonically isomorphic to the rotation algebra. Indeed, the functions $u,v\colon \mathbb{Z} \to \mathbb{C}$ given by
    \begin{align}
        u(\vec{g})=\delta_{g_{1},1}\delta_{g_{2},0}
        \quad\text{ and }\quad
        v(\vec{g}) = \delta_{g_{1},0}\delta_{g_{2},1}
        \quad\text{ satisfy the relation }vu = \lambda uv,
    \end{align}
    and the map $U\mapsto u,V\mapsto v$ extends to a $*$-isomorphism $A_{\theta}\cong \csr(G;\mathbb{T}\times_{\cocycle_{\theta}}G)$.
    
\end{example}

\begin{example}[Rotation algebra as action groupoid $\cs$-algebra]\label{ex:rotation as transformation}
    We let $G=\mathbb{Z}$ act on $X=\mathbb{T}$ by $n\HleftX z = \lambda^n z$, where as before $\lambda=\mathsf{e}^{2\pi i\theta}$ for a fixed $\theta\in\mathbb{R}$. We denote the corresponding action groupoid by
    $\actiongpd[\theta]{\mathbb{T}}{\mathbb{Z}}$. Its (untwisted) $\cs$-algebra $\csr(\actiongpd[\theta]{\mathbb{T}}{\mathbb{Z}})$ is canonically isomorphic to the rotation algebra: the functions $u,v\colon \actiongpd[\theta]{\mathbb{T}}{\mathbb{Z}} \to \mathbb{C}$ given by
    \begin{align}
        u(n,z)=\delta_{n,0}z
        \quad\text{ and }\quad
        v(n,z) = \delta_{n,1}
        \quad\text{ satisfy the relation }vu = \lambda uv,
    \end{align}
    and the map $U\mapsto u,V\mapsto v$ extends to a $*$-isomorphism $A_{\theta}\cong \csr(\actiongpd[\theta]{\mathbb{T}}{\mathbb{Z}})$.
\end{example}

We will denote the
\emph{open support} of a function $f$ on $\E$ by $\suppo(f)$, while $\supp(f)$
will denote its closure, and we further write
\[
    \suppoG(f)\coloneqq  \pi(\suppo(f)).
\]
We use this moment to record a well-known, easily proved result that we will need later on. 
\begin{lemma}\label{lem:unique positive convergent lift}
    Suppose $f\in C_{c}(G; \E  )$ is such that $\suppoG(f)$ is a bisection. Suppose that $u_i\to u$ is a convergent net in  $s(\suppo(f))$. Then there exist unique $ e, e_{i}\in \suppo(f)$ with $s( e_i)=u_i$, $s( e)=u$, and $f( e_i),f( e)>0$, and this net satisfies $ e_i\to  e$.
\end{lemma}

\subsection{Cartan pairs and Weyl pairs}\label{ssec:Cartan}
Next, let us go into more detail regarding Kumjian--Renault theory.
\begin{defn}    [{\cite[Definitions 2.7 and 2.11]{pitts2022normalizers}}] A subalgebra $B$ of a $\cs$-algebra $A$ is called a {\em Cartan subalgebra} if 
\begin{enumerate}
    \item $B$ is maximal abelian in $A$;
    \item there exists a faithful conditional expectation $A\to B$; and 
    \item $B$ is {\em regular} in the sense that the set
    \[
N(B) \coloneqq   \{m\in A: m^* B m, mBm^*\subset B\}
\] of normalizers of $B$ in $A$, generates $A$ as a $\cs$-algebra.
\end{enumerate}
A subalgebra $B$ is called a {\em $\cs$-diagonal} of $A$ if it is a Cartan subalgebra such that
\begin{enumerate}[resume]
    \item every pure state of $B$ uniquely extends to a state of $A$.
\end{enumerate}
\end{defn}
The original definitions for Cartan subalgebras and $\cs$-diagonals can be found in \cite[Definition 5.1]{Renault:Cartan} and \cite[Definition 1.3$^\circ$]{Kum:Diags}, respectively; the definitions stated above are equivalent by \cite{pitts2022normalizers}, Theorem 2.5 and Proposition 2.9, respectively.

As mentioned in the introduction, every twist $\E$ over an effective groupoid~$G$ gives rise to a Cartan pair $C_0(G\z)\subset \csr(G;\E)$ \cite[Theorem 5.2]{Renault:Cartan}, and conversely, every Cartan pair $B\subset A$ gives rise to an effective groupoid with twist in such a way that there exists a Cartan-preserving isomorphism from $A$ to the reduced twisted groupoid $\cs$-algebra \cite[Theorem 5.9]{Renault:Cartan}.\footnote{Renault focused on the situation in which $G$ is second countable, i.e., $A$ is separable. This result has 
since been extended to the non-separable / non-second countable case; see \cite[Theorem 1.2]{Raad:2022:Renault} and \cite[Corollary 7.6]{KM:2020:NCCartans}.} Moreover, the groupoid~$G$ being principal rather than just effective corresponds to $B$ being a $\cs$-diagonal rather than just a Cartan subalgebra \cite[Theorems 2.9$^{\circ}$ and 3.1$^{\circ}$]{Kum:Diags}. As we will be interested in identifying the Weyl twist and groupoid associated to a non-traditional Cartan pair, we will briefly recall how these groupoids can be constructed from the inclusion $B\subset A$.

For a normalizer $n\in N(B)$
of $B$ in $A$,
we let
\[
\dom(n)\coloneqq  \{x\in \widehat{B}: x(n^*n)\neq 0\}.
\]
Here, we have made use of the fact that $B$ contains an approximate identity for $A$ \cite[Corollary 2.6]{pitts2022normalizers}, so that $n^*n\in B$ and $x(n^*n)$ makes sense.
We let
$\alpha_{n}\colon\dom(n)\to\dom(n^*)$ denote the partial homeomorphism of $\widehat{B}$ from \cite{Kum:Diags} that is determined by the equality
\begin{align}\label{eq:def'n of alpha_n}
    x(n^* b n) = \alpha_{n}(x)(b) \, x(n^*n)
    \quad\text{ for all } b\in B \quad\text{ and }\quad x\in\dom(n).
\end{align}
On the set
\begin{equation*}\label{eq:E}
    \left\{(n,x)\in N(B)\times \widehat{B}: x\in \dom(n)\right\},
\end{equation*}
we define two equivalence relations:
\begin{align*}
(n,x) \sim (m,x) \text{ if } &\text{there exist }b,b'\in B
\text{ such that } x(b),x(b')\neq 0\text{ and } nb=mb';
\\
(n,x) \approx (m,x) \text{ if } &\text{there exist }b,b'\in B
\text{ such that }x(b),x(b')>0\text{ and } nb=mb'.
\end{align*}
The quotient by $\approx$ is the so-called {\em Weyl twist} $\Sigma_{B\subset A}$, and its elements will be denoted by $\llbracket n,x\rrbracket$. We give it the following groupoid structure:
\[
    \llbracket m,\alpha_n(x)\rrbracket\llbracket n,x\rrbracket=\llbracket mn,x\rrbracket
    \quad\text{and}\quad
    \llbracket n,x\rrbracket\inv = \llbracket n^*,\alpha_n(x)\rrbracket.
\]
The topology is generated by basic open sets of the form
\[
	\{\llbracket zn, x\rrbracket : z \in U, x\in V\cap \dom(n) \}
    \quad\text{ for }
    n \in N(B)
    \text{ and open }
    U \subset\mathbb{T},
    V\subset\widehat{B}
    .
\]
If we instead quotient out by $\sim$, then we arrive at the {\em Weyl groupoid} $W_{B\subset A}$, whose elements we will denote by $[n,x]$.\footnote{\cite[Proposition 2.2]{DGN:Weyl} proves that our definition of ${\sim}$ coincides with the definition in \cite{Renault:Cartan}.} The groupoid structure and the topology of $W_{B\subset A}$ are defined such that the map
\begin{equation}\label{eq:def:Pi}
\Pi\colon \Sigma_{B\subset A}\to W_{B\subset A},
\quad \llbracket n, x\rrbracket\mapsto [n,x],
\end{equation}
is a continuous (and automatically open) homomorphism.  Note that the unit space of $\Sigma_{B\subset A}$ and $W_{B\subset A}$ is given by $\widehat{B}$.
For each $(z,x)\in\mathbb{T}\times \widehat{B}$, fix an arbitrary element $b_{z,x}\in B$ with $x(b_{z,x})=z$. Together with $\Pi$, the map
\begin{equation}\label{eq:def:iota}
\iota\colon \mathbb{T}\times \widehat{B}
\to \Sigma_{B\subset A},
\quad (z,x)\mapsto \llbracket b_{z,x}, x\rrbracket
\end{equation}
realizes $\Sigma_{B\subset A}$ as a twist over $W_{B\subset A}$. More details can be found in \cite{Renault:Cartan}.

\subsection{Actions on the twisted dual of~$S$}\label{ssec:actions}
As mentioned in the introduction, our main results revolve around action groupoids built from actions of $\E$ on $\twPD$, so let us remind the reader of the definition of $\twPD$ and of these actions.
\begin{assumption}\label{as:E,G,S}
   $G$ is a \LCH\ groupoid,  $\pi\colon\E\to G$ is a twist, and~$S$ is a closed 
    and normal subgroupoid of~$G$ with $S\z$ equal to $\cU\coloneqq  G\z$ and for which the groupoid $ \E_{S} \coloneqq  \pi\inv(S)$ is abelian. Lastly, we assume that the range map of~$S$  is open. 
\end{assumption}

\begin{remark}
    \begin{enumerate}
        \item 
    We do not need to assume that~$G$ be second countable or \etale, or to have open range map. That the range map of $S$ be open is needed to ensure that the quotient map $G\to G/S$ is open; see Lemma~\ref{lem:quotients}\ref{item:quotient map open}.
    \item
    In the setting of Assumption~\ref{as:E,G,S}, $ \E_{S} $ is normal and closed in $\E$,~$S$ is abelian, and $(G,\E,S)$ is in particular an {\em abelian twisted extension} in the sense of \cite[Definition 3.4]{Renault:2023:Ext}. 
    \item
    If a groupoid has a Haar system, then its range map is automatically open \cite[Proposition 1.23]{Wil2019}. 
     Thus, an easy way to make sure that~$S$ has open range map is to assume that~$G$ is \etale\ and that~$S$ is open in~$G$ \cite[Proposition 1.29]{Wil2019}. This is the situation we will be most interested in in later sections.
     \item If~$G$ has a Haar system and is second countable, then it was shown in \cite[Lemma 2.1]{IKRSW:2021:Extensions} that the closed subgroup bundle~$S$ has open range map if and only if it also has a Haar system.
    \end{enumerate}
\end{remark}

\begin{defn}[cf.\ {\cite[Definition 3.6]{Renault:2023:Ext}}]\label{def:twPD}
    For $u\in \cU  $, write $\widehat{\E}_{S} (u)$ for the Pontryagin dual of the abelian group $\E_{S}(u)$. Let
    $P\colon \bigsqcup_{u\in\cU } \widehat{\E}_{S} (u) \to \cU  $ be the map that sends $\widehat{\E}_{S} (u)$ to $\{u\}$, and let $\twPD(u)$ denote the subset of $\widehat{\E}_{S} (u)$ that consists of those continuous homomorphisms $\kappa\colon  \E_{S} (u)\to\mathbb{T}$ for which
    $(\kappa\circ \iota)(z,u) = z$.
    The \emph{$\E_{S} $-twisted Pontryagin dual} (or just \emph{twisted dual}) of~$S$ is the set
    \[
        \twPD\coloneqq  
        \bigsqcup_{u\in\cU } \twPD(u)
        =
        \bigsqcup_{u\in\cU }
        \bigl\{
            \kappa\in  \widehat{\E}_{S} (u): \kappa(z\cdot\sigma)=z\kappa(\sigma) \text{ for all }z\in\mathbb{T},\sigma\in \E_{S} (u)
        \bigr\}
    \]
    equipped with the unique topology such that a net $(\kappa_{i})_{i}$ converges to $\kappa$ if and only if the following two conditions hold:
    \begin{enumerate}[label=\textup{(\roman*)}]
        \item\label{item:convergence of units} $u_{i}\coloneqq  P(\kappa_{i})\to u\coloneqq  P(\kappa)$ in $\cU  $; and
        \item\label{item:convergence netwise} if  $(\tau_{i})_{i}$ is a lift of $(u_{i})_{i}$ in $\E_{S}$ under the source map and if $\tau_{i}\to \tau$, then $\kappa_{i}(\tau_{i})\to \kappa(\tau)$.
    \end{enumerate}
\end{defn}

\begin{remark}\label{rmk:twPD for no twist}
    At times, we will be in the situation that $\E_{S}$ is trivializable. In this case, $\twPD$ is canonically homeomorphic to \(
        \widehat{S}\coloneqq   \bigsqcup_{u\in \cU } \widehat{S(u)}
    \)
    with the topology as in \cite[Proposition 3.3]{MRW:1996:CtsTrace}.
    Indeed, if $\mfs\colon S\to \E_{S}$ is a continuous and homomorphic section of $\pi|_{S}\colon \E_{S}\to S$, then the map
    $\mfs^*\colon \twPD\to\widehat{S}$ which sends $\kappa\in\twPD(u)$ to the element
    \[
        \mfs^*(\kappa)\colon\quad
        t\mapsto \kappa(\mfs(t))
    \]
    of $\widehat{S(u)}$, is a homeomorphism.
\end{remark}

For $e\in v \E u $ and $\tau\in  \E_{S} (v)$, normality of~$S$ implies that $e\inv \tau e \in  \E_{S} (u)$, so if  $\kappa\in \twPD(u)$, then $\kappa (e\inv \tau e)$ makes sense. This serves as a first sanity check for the following proposition.

\begin{prop}[{\cite[p.\ 22ff]{IKRSW:2021:Extensions}, \cite[p.\ 259ff]{Renault:2023:Ext}}]\label{prop:E:HleftX is an action}
    The map $\mvisiblespace \HleftX \mvisiblespace\colon \E\bfp{s}{P} \twPD \to \twPD $ given for $e\in  v \E u $, $\tau\in  \E_{S} (v)$, and $\kappa\in \twPD(u)$ by
    \[
        \left( e\HleftX \kappa\right) (\tau)
        \coloneqq  
        \kappa (e\inv \tau e),
    \]
    defines a continuous action of $\E$ on $\twPD$.
\end{prop}

Note that, when restricted to $ \E_{S} $, the action is trivial.

\begin{proof}
	It is straight forward to check that $e\HleftX\kappa$ is an element of $\twPD(v)$ and that $(e,\kappa)\mapsto e\HleftX \kappa$ satisfies the algebraic properties of an action.
   To see that 
    it 
    is  continuous, assume that we have a net $( e_{i},\kappa_{i})_{i}$ that converges to $( e ,\kappa)$ in $\E\bfp{s}{P} \twPD $. 
    Since 
    \[
    P( e_{i}\HleftX \kappa_{i})= r( e_{i})\to 
    r( e )=P( e \HleftX \kappa),\]
    we see that Condition~\ref{item:convergence of units} of  Definition~\ref{def:twPD} holds. For Condition~\ref{item:convergence netwise}, assume that $(\tau_{i})_{i}$ is a lift of $(r(e_{i}))_{i}$ under the source map of $\E_{S}$ and that $\tau_{i}\to \tau$.
     It follows from Condition~\ref{item:convergence netwise} applied to the convergent net $\kappa_{i}\to \kappa$ that 
    $\kappa_{i}(e_{i}\inv \tau_{i} e_{i}) \to \kappa(e\inv\tau e)$, which means exactly that $ e_{i}\HleftX \kappa_{i}\to  e \HleftX \kappa$.
\end{proof}

In light of Theorem~\ref{thm:A}, we see that we will need an action of the quotient $G/S$ on $\twPD$, so let us make sense of this.
We turn $\E$ into a right $ \E_{S} $-space by defining $\E\bfpsr \E_{S} \to \E$ by $(e,\sigma)\mapsto e\sigma$. Elements of the quotient $\E/ \E_{S} $ are then of the form \[\dot{e}\coloneqq   e\E_{S}=\{e\sigma:\sigma\in \E_{S} (s(e))\}=\{\tau e : \tau\in  \E_{S} (r(e))\}\] for $e\in\E$. 
Fully analogously, we define $G/S$. To show that the action of $\E$ on $\twPD$ passes to an action of $G/S $, we first need a lemma that ensures that we will not run into any topological issues.

\begin{lemma}\label{lem:quotients}
    In the setting of Assumption~\ref{as:E,G,S}, the following hold.
\begin{enumerate}[label=\textup{(\arabic*)}]
    \item\label{item:quotient map open} The quotient maps $q_{\E}\colon \E\to\E/ \E_{S} $ and $q_{G}\colon G\to G/S$ are open.
    \item\label{item:Lambda LCH}
    $\E/ \E_{S} $ and $G/S$ are \LCH\ groupoids.
    \item\label{item:Lambda is Q} The map $\pi\colon \E\to G$ descends to an isomorphism $\tilde{\pi}\colon\E/ \E_{S} \cong G/S$ of topological groupoids.
\end{enumerate}
\end{lemma}
Because of Part~\ref{item:Lambda is Q}, we will frequently denote elements of $G/S$ by $\dot{e}=q_{\E}(e)$ for $e\in \E$.
\begin{proof}
    \ref{item:quotient map open} 
    Since we assumed that $S\z=G\z$, that the range map of~$S$ is open, and that~$S$ is closed (and hence locally closed) in the locally compact groupoid~$G$, it follows from \cite[Corollary A.8]{Renault:2023:Ext} that $q_{G}$ is open. These  assumptions on~$S$ likewise imply that the range map of $\E_{S}$ is open 
    and that $\E_{S}$ is locally closed in $\E$, so the same result implies that $q_{\E}$ is open. (Alternatively, one can easily prove \ref{item:quotient map open} using Fell's criterion \cite[Proposition 1.1]{Wil2019}.)
    
    \ref{item:Lambda LCH} That $\E/ \E_{S} $ is a groupoid is standard (see \cite{AMP:IsoThmsGpds}). By \cite[Example 2.16]{Wil2019}, the action of $\E$ on itself is proper. Since~$S$ is closed, this implies that the action of $ \E_{S} $ on $\E$ is likewise proper.
    By \cite[Proposition 2.18]{Wil2019}, $\E/ \E_{S} $ is thus locally compact and Hausdorff. Multiplication and inversion are continuous because the quotient map is open and continuous, and because multiplication and inversion are continuous on $\E$.

    \ref{item:Lambda is Q} 
    It is clear that $\tilde{\pi}$ is an isomorphism of algebraic groupoids. To see that it is a homeomorphism, one can make repeated use of Fell's criterion \cite[Proposition 1.1]{Wil2019}; 
    that way, one gets that $\tilde{\pi}$ is continuous because $q_{\E}$ is open and because $\pi$ and $q_{G}$ are continuous; and one gets that $\tilde{\pi}$ is open because $q_{\E}$ is continuous and  because $\pi$ and $q_{G}$ are open.
\end{proof}

\begin{corollary}
\label{cor:G/S:HleftX is an action}
    The action of $\E$ on $\twPD$ defined in Proposition~\ref{prop:E:HleftX is an action} descends to a continuous action $\mvisiblespace \HleftX \mvisiblespace\colon (G/S)\bfp{s}{P} \twPD \to \twPD $ given by the formula
    \[
        \left( \dot{e}\HleftX \kappa\right) (\tau)
        \coloneqq  
        \kappa (e\inv \tau e).
    \]
\end{corollary}

\begin{proof}
    If $f=e\sigma$ for some $\sigma\in \E_{S} (s(e))$, then
    \[
    \kappa ( f\inv \tau f)
    =
    \kappa \bigl((e\sigma)\inv \tau (e\sigma)\bigr)
    =
    \kappa (\sigma\inv (e\inv \tau e)\sigma)
    =
    \kappa(e\inv \tau e),
    \]
    In other words 
    $\dot{e}\HleftX \kappa$ does not depend on the chosen representative $e$ of $\dot{e}$.
    Lastly, using Fell's criterion, continuity of the induced action follows from the fact that the original action is continuous and that $q_{\E}$ is an open map (Lemma~\ref{lem:quotients}\ref{item:quotient map open}).
\end{proof}

The results of Proposition~\ref{prop:E:HleftX is an action} and of Corollary~\ref{cor:G/S:HleftX is an action} below were used in \cite[p.\ 22ff]{IKRSW:2021:Extensions} and \cite[p.\ 259ff]{Renault:2023:Ext}. 
In both papers, the authors assumed that $\E$ has a Haar system so that, in particular, its range map is open \cite[Proposition 1.23]{Wil2019}, and that the closed subgroup bundle~$S$ has a Haar system. As the proofs given here show, these assumptions are needed to ensure that the action of $\E/ \E_{S} $ on $\twPD$ is continuous.

\begin{hooptedoodle}\label{hoop:G also acts on twPD}
    Since the $\mathbb{T}$-action on $\E$ is central and since the map $\pi\colon \E\to G$ is also an open map, the same argument as in the proof of Corollary~\ref{cor:G/S:HleftX is an action} shows that we also get a continuous action of $G=\E/\mathbb{T}$ on $\twPD$ given by $\pi(e)\HleftX \kappa \coloneqq   e\HleftX\kappa$. 
\end{hooptedoodle}

Since the groupoids $G/S$ and $\E/\E_{S}$ are canonically isomorphic (Lemma~\ref{lem:quotients}), one has to be careful not to misinterpret the action of that quotient on $\twPD$; the following example illustrates this.

\begin{example}[continuation of Example~\ref{ex:c_theta:setup}]\label{ex:c_theta:S=Zx0}
    We return to the twisted group $\mathbb{Z}^2$ with $2$-cocycle~$\cocycle_{\theta}$ defined at~\eqref{eq:c_theta}.  Let us record the following equality in the twist $\E_{\theta}\coloneqq\mathbb{T}\times_{\cocycle_\theta}\mathbb{Z}^2$, which we will make frequent use of: if $ \vec{g} , \vec{h} \in \mathbb{Z}^{2}$ and $z,w\in\mathbb{T}$, then
    \begin{equation}\label{eq:c_theta:conjugation}
        ( w;  \vec{g} )\inv (z;\vec{h} ) ( w;  \vec{g} ) 
        =
        \bigl(
            \lambda^{g_{1}h_{2} - g_{2}h_{1}}
            z
            ; \vec{h} 
        \bigr)
        \quad
        \text{where }\lambda=\mathsf{e}^{2\pi i\theta}.
    \end{equation}
    
    The $2$-cocycle $\cocycle_{\theta}$ is constant $1$ on the normal subgroup $S\coloneqq  \{0\}\times\mathbb{Z}$ of $G=\mathbb{Z}^{2}$, so that $\E_{\theta,S}\coloneqq  (\E_{\theta})_{S}$ is the trivial twist, $\mathbb{T}\times S$. It follows directly from Equation~\eqref{eq:c_theta:conjugation} that an element $\vec{g}=(g_{1},g_{2}) \in G$ acts on an element $\kappa\in \twPD[\E_{\theta}]$ as per Hooptedoodle~\ref{hoop:G also acts on twPD} by
    \[
     \left(\vec{g}\HleftX\kappa\right) ( z; 0,n)
     =
     \kappa
     \left(
     \lambda^{g_{1}n}
     z; n,0\right)
     \quad\text{ for }( z;0,n)\in \E_{\theta,S}.
    \]
    Since $\E_{\theta,S}$ is trivial,  $\twPD[\E_{\theta}]$ is canonically homeomorphic to $\widehat{S}$ (Remark~\ref{rmk:twPD for no twist}) which, via the `normal' Fourier transform, is $\mathbb{T}$. To be more explicit, $z\in\mathbb{T}$ corresponds to the element $\kappa_{z}\in \twPD[\E_{\theta}] $ that sends $\left( w; 0,n\right)\in\E_{\theta,S}$ to $w z^n\in \mathbb{T}$. The above can then be rewritten to
    \[
    \vec{g}\HleftX
    z
    = 
    \lambda^{g_{1}}
    z
    \quad\text{ where }\quad z\in\mathbb{T}.
    \]
    Thus, if we further identify $G/S\cong\E_{\theta}/\E_{\theta,S}$ with $\mathbb{Z}$ by $(g_{1},g_{2})+S\mapsto g_{1}$, then we arrive at the following action of $\mathbb{Z}$ on $\mathbb{T}$:
    \[
     n\HleftX z
     =
     \lambda^{n}
     z
     \quad\text{ where }\quad n\in\mathbb{Z},z\in\mathbb{T}.
    \] 
    In other words, the action groupoid $\actiongpd{\twPD[\E_{\theta}]}{(G/S)}$ of the action in Corollary~\ref{cor:G/S:HleftX is an action} is canonically isomorphic to the action groupoid $\actiongpd[\theta]{\mathbb{T}}{\mathbb{Z}}$ of rotation by $\theta$ on the circle as given in Example~\ref{ex:rotation as transformation}.
    This example shows that, even if $\E_{S}$ is trivial, then the action of $G/S$ on $\widehat{S}$ is \emph{not} given by a formula such as $(\dot{g}\HleftX \chi)(s)=\chi(g\inv sg)$ for $\chi\in\widehat{S}$, for otherwise any {\em abelian}~$G$ like $\mathbb{Z}^{2}$ above would induce a trivial action of $G/S$ on $\widehat{S}$. 
\end{example}

\section{The Weyl pair from a non-traditional Cartan pair}\label{sec:Weyl}

\begin{assumption}\label{as:E,G,S,v2}
    From now on, we assume the conditions on $\atgext$ of Theorem~\ref{thm:DWZ}\ref{it:thm:DWZ:S}, that is:
    \begin{itemize}
        \item $G$ is a \LCH, \etale\ groupoid;
        \item $\E$ is a twist over~$G$ as in~\eqref{eq:def twist};
        \item $S$ is a clopen normal subgroupoid of~$G$ with $S\z$ equal to $\cU\coloneqq G\z$; and
        \item $S$ satisfies the conditions~\eqref{eq:thm:DWZ:maximal}  and~\eqref{eq:thm:DWZ:ricc} of Theorem~\ref{thm:DWZ}.
    \end{itemize}
\end{assumption}

\begin{remark}
    \begin{enumerate}
        \item If~$S$ is any \LCH, \etale\ groupoid with {\em abelian} twist $\E_{S}$, then choosing $G=S$ and $\E=\E_{S}$ gives a twist $\atgext$ as in Assumption~\ref{as:E,G,S,v2}.

    \item    Since~$S$ is open and~$G$ \etale, the source map of~$S$ is open. Moreover, it follows from ~\eqref{eq:thm:DWZ:maximal} that $\E_{S}$ is abelian \cite[Corollary 3.17]{DWZ:2025:Twist}. Thus, Assumption~\ref{as:E,G,S,v2} implies Assumption~\ref{as:E,G,S}, so we may invoke our results from Subsection~\ref{ssec:actions}. 
    \item There is still no need for us to assume that $G$ is second countable, as we will not make use of Renault's disintegration theorem.
    \end{enumerate}
\end{remark}

\begin{notation}
By Theorem~\ref{thm:DWZ}, $B\coloneqq   \csr(S; \E_{S} )$ is a (non-traditional) Cartan subalgebra of $A\coloneqq   \csr(G;\E)$.
     We denote by 
     \begin{itemize}
         \item $\Sigma_{B\subset A}\to W_{B\subset A}$ the associated Weyl twist and groupoid;
         \item $\twPD$ the twisted dual of~$S$ (Definition~\ref{def:twPD}); and
         \item 
         $\actiongpd{\twPD}{\E} $ and          $\Weylgpd$ the action groupoids
         as in Definition~\ref{def:action groupoid} of the actions in Proposition~\ref{prop:E:HleftX is an action} and Corollary~\ref{cor:G/S:HleftX is an action}, respectively.\footnote{In \cite{Renault:2023:Ext}, Renault uses the symbol $\twPD\rtimes\E$ instead, since he defined the action of $\E$ on $\twPD$ as a right action \cite[p.\ 259]{Renault:2023:Ext}.} 
     \end{itemize}
\end{notation}

The goal of this section is to describe the above Weyl pair entirely in terms of the original data $\atgext $.
To this end, we first identify the spectrum $\widehat{B}$ of $B$ in terms of the underlying abelian twist $ \E_{S}\to S$:

\begin{prop}[{\cite[Proposition 3.7]{Renault:2023:Ext}}]
\label{prop:spec B}
    For $u\in \cU  $ and $\kappa\in \twPD(u)$, the map $\intA{\kappa}\colon C_c(S; \E_{S} )\to\mathbb{C}$ given by\footnote{In Definition~\ref{def:twPD}, we could have defined $\twPD(u)$ instead as those $\kappa\in\widehat{\E}_{S}(u)$ that satisfy $\kappa(z\cdot\sigma)=\overline{z}\kappa(\sigma)$. In this case, one would here need to use the Fourier transform as stated in \cite[Section 3]{MRW:1996:CtsTrace} instead of Equation~\eqref{eq:phi_kappa}. (Equation~\eqref{eq:phi_kappa} coincides with \cite[Equation (3.1)]{IKRSW:2021:Extensions}.)}
    \begin{align}\label{eq:phi_kappa}
        \intA{\kappa} (f) &= \sum_{   \{ \pi(\sigma)\in S(u) \}   }\overline{\kappa(\sigma)} f( \sigma)
    \end{align}
    is a $*$-algebra homomorphism and extends to a map on $\csr(S; \E_{S} )$. Moreover, the map 
    \[\sharp\colon\twPD \to \bigl(\csr(S; \E_{S} )\bigr)\,\widehat{}\;, \quad \kappa\mapsto \intA{\kappa},\] is a homeomorphism. In particular, the formula
\[
	e\HleftX (\intA{\kappa}) \coloneqq   \intA{(e\HleftX \kappa)}
\]
defines an action of $\E$, $G$, $G/S$ on $\widehat{B}$.
\end{prop}

 We can now state our first theorem which concerns the Weyl groupoid,  $W_{B\subset A}$.

\begin{thm}\label{thm:Weyl groupoid}
There exists an isomorphism $\delta\colon \actiongpd{\twPD}{(\E/\E_{S})} = \Weylgpd   \to W_{B\subset A}$ of topological groupoids determined by
    \[
    \delta (\dot{e},\kappa) = 
     [ 
        n, \intA{\kappa}
     ] 
    \]
    where $n\in C_c (G; \E )$ is any element with $\suppoG(n)$ a bisection and with $\dot{e}\in q_{\E}(\suppo(n))$.
\end{thm}

The second theorem shows that the Weyl twist, $\Sigma_{B\subset A}$, is canonically isomorphic to a quotient of  
         $\actiongpd{\twPD}{\E} $:

\begin{thm}\label{thm:Weyl twist}
    There exists a surjective, continuous, open groupoid homomorphism
    $\tilde{\Delta} \colon \actiongpd{\twPD}{\E}  \to \Sigma_{B\subset A}$ determined by
    \[
    (e,\kappa) 
    \longmapsto
     \llbracket 
        n, \intA{\kappa}
     \rrbracket
    \]
    where $n\in C_c (G; \E )$ is any element with $\suppoG(n)$ a bisection and with $n(e)>0$. Its kernel is given by 
     the subgroupoid
    \begin{equation}\label{eq:def cK}
        \cK \coloneqq  
        \left\{(\sigma,\kappa)\in \actiongpd{\twPD}{\E} :
        \sigma\in \E_{S} , \kappa(\sigma)=1\right\}.
    \end{equation}
    and the induced groupoid isomorphism
    \[
    \Delta\colon 
    \bigl(\actiongpd{\twPD}{\E}\bigr)/\cK
    \overset{\cong}{\longrightarrow} \Sigma_{B\subset A}
    \]
    is a homeomorphism.
\end{thm}
    
\begin{remark}\label{rmk:Weyl twist}    
 The quotient $\bigl(\actiongpd{\twPD}{\E}\bigr)/\cK$  is denoted by $\underline{\Sigma}$ in \cite{Renault:2023:Ext}, by $\widehat{\Sigma}$ in \cite[Section 3.1]{IKRSW:2021:Extensions}, and by $\tilde{\Sigma}$ in \cite[Section 3.1]{IKRSW:2021:Pushout}.
    Going forward, we will instead use the non-standard notation
    \[
    \Weyltwist\coloneqq\bigl(\actiongpd{\twPD}{\E}\bigr)/\cK.
    \]
    To explain this choice, fix $u\in\cU=G\z$. An arbitrary element of the quotient can be written as follows, where $e\in \E_{u}$ and $\kappa\in \twPD(u)$:
    \begin{align*}
        (e,\kappa)\cK
        =
        \{(e\sigma,\kappa): \sigma\in \ker(\kappa)\}
    \end{align*}
    Now, for any $s\in S(u)$, there there exists a unique lift $\sigma$ of $s$ under $\pi$ such that $\kappa(\sigma)=1$. Indeed, if $\tau$ is any lift, then $\sigma\coloneqq   \overline{\kappa(\tau)}\cdot \tau$ has this property since $\kappa$ is $\mathbb{T}$-equivariant; in particular, we may also write the coset as
    \begin{align}\label{eq:cosets of cK}
        (e,\kappa)\cK
        =
        \{(\overline{\kappa(\tau)}\cdot e\tau,\kappa): \tau\in \E_{S}(u)\}
        .
    \end{align}
    Moreover, if $\sigma'$ is any other lift of $s$, then there exists $z\in\mathbb{T}$ with $\sigma'=z\cdot \sigma$, and so
    \begin{equation}\label{eq:unique lift wrt kappa}
        \kappa(\sigma')=\kappa(z\cdot \sigma)=z\kappa(\sigma)=z,
    \end{equation}
    showing that $\sigma$ is unique. 
    So in the quotient $(\actiongpd{\twPD}{\E} )/\cK$, we are, in essence, ``modding out not all of $\E_{S}$ but only~$S$.'' 
\end{remark}

\begin{remark}[Caveat]\label{rmk:confusing}
    In Theorems~\ref{thm:Weyl groupoid} and~\ref{thm:Weyl twist} (and all of their variants below), it is important to keep in mind that $\twPD$ is really just a space, not a groupoid, so that $\actiongpd{\twPD}{\E}$ (and its variants) is an action groupoid in the sense of Definition~\ref{def:action groupoid} and not an honest-to-goodness semi-direct product $\actiongpd{\mathcal{H}}{\mathcal{G}}$ of two groupoids. The reason we point this out is that matters get confusing when $\E_{S}$ happens to be trivializable: in this case, the twisted dual $\twPD$ can be identified with the dual $\widehat{S}$ (Remark~\ref{rmk:twPD for no twist}), so it inherits the structure of a group bundle with unit space $S\z$. The symbol $\actiongpd{\twPD}{\E}$ could therefore be misinterpreted to mean that $\E$ acts on {\em the groupoid} $\twPD$ by homomorphisms.

    There are a few ways for the reader to remember that, for our purposes, we must treat $\twPD$ as a space:
    (1)~Even if $\twPD$ inherits a groupoid structure from $\widehat{S}$, the action of $\E$ that we define on $\twPD$ in Proposition~\ref{prop:E:HleftX is an action} can be verified to {\em not} be by homomorphisms, so the semi-direct product would not have a well-defined structure of a groupoid, meaning that Theorem~\ref{thm:Weyl groupoid} would be nonsense. (2)~For Theorem~\ref{thm:Weyl groupoid} to have a chance of being true, we need the unit space of $\Weylgpd$ to coincide with the Gelfand dual of $B=\csr(S;\E_{S})$ which we know to be $\twPD$ (Proposition~\ref{prop:spec B}); the unit space of a semi-direct product groupoid $\actiongpd{\mathcal{H}}{\mathcal{G}}$ is, however, the unit space $\mathcal{H}\z$ of $\mathcal{H}$ \cite[Remark 2.20]{DuLi:ImprimThms}, not $\mathcal{H}$ itself.

    We should nevertheless keep in mind $\widehat{S}$'s `dual nature' (pun intended): When analyzing which groupoids $W$ can arise as the Weyl groupoid of certain non-traditional Cartan pairs, the authors of \cite{BG:2023:Gamma-Cartan-pp} needed to make use of $\widehat{S}$ both as a space in action groupoids {\em and} as a groupoid in its own right in semi-direct (and Zappa--Sz\'ep) products.
\end{remark}

\begin{remark}
    The above theorems are generalizations of those found in \cite{DGN:Weyl}, where we focused on twists $\E$ that are induced from a $2$-cocycle on a second countable groupoid~$G$ and for which the quotient map $q_{G}\colon G\to G/S$ allows a continuous section. To be precise,
    Theorem~\ref{thm:Weyl groupoid} generalizes \cite[Theorem 4.6]{DGN:Weyl}, and Theorem~\ref{thm:Weyl twist}  generalizes \cite[Theorem 5.1, Corollary 5.4]{DGN:Weyl}.
\end{remark}

To prove the theorems, we first need to better understand Renault's description of the Weyl groupoid and twist in the specific setting of our non-traditional Cartan pair, so that we can translate its properties into properties of the groupoids 
$\Weyltwist$ and $\Weylgpd$; see Proposition~\ref{prop:Weyl explained}. To this end, we generalize the ideas found in \cite[Section 4]{DGN:Weyl}:

Since~$S$ is normal, it follows from \cite[Corollary 3.25]{DWZ:2025:Twist} that
\begin{equation}\label{eq:def of N}
    N\coloneqq   \{ n\in C_c (G; \E ) : \suppoG(n)\text{ is a bisection}\}
\end{equation}
is a subset of $N(B)$. Since $N$ generates $A$ as a $\cs$-algebra, it follows from \cite[Proposition 4.1]{DGNRW:Cartan}, \cite{BG:2022:DGNRW}, \cite[Lemma 6.13.]{DL:ZStwist-pp} that every element of the Weyl twist (resp.\ groupoid) can be written as $\llbracket n,x\rrbracket$ (resp.\ $[n,x]$) for some $n\in N$ and some $x\in \dom(n)$.\footnote{While the groupoids in \cite{DGNRW:Cartan} are assumed to be second countable, it is clear that the proof given in \cite{BG:2022:DGNRW} does not depend on this assumption.} This is why, in the following, we restrict our attention to normalizers from this subset.

\begin{lemma}\label{lem:DGN Proposition 4.4}
    Suppose $n\in N$, $u\in \cU $, and $\kappa\in \twPD(u)$.
    \begin{enumerate}[label=\textup{(\arabic*)}]
        \item\label{item:DGN Proposition 4.4(1)} $u\in s(\suppo(n))$ if and only if $\intA{\kappa}\in \dom(n)$.
        \item\label{item:DGN Proposition 4.4(2)} If $e\in \suppo(n)\cap \E u $, then $\alpha_{n}(\intA{\kappa})=\intA{e\HleftX \kappa}$, where $\alpha_n$ is the partial homeomorphism determined by~\eqref{eq:def'n of alpha_n}.
    \end{enumerate}
\end{lemma}

\begin{proof}
    The proof of Part~\ref{item:DGN Proposition 4.4(1)} is verbatim that for  the case of a $2$-cocycle in \cite[Proposition 4.4(1)]{DGN:Weyl}, just that the reference to \cite[Lemma 3.11]{DGNRW:Cartan} has to be replaced with \cite[Corollary 3.23]{DWZ:2025:Twist}.

    To prove Part~\ref{item:DGN Proposition 4.4(2)}, we follow the ideas of \cite[Proposition 4.4(2)]{DGN:Weyl}. It follows from Part~\ref{item:DGN Proposition 4.4(1)} that $\intA{\kappa}\in\dom(n)$, so that $\alpha_{n}(\intA{\kappa})$ is defined and an element of $\dom(n^*)$. 
     Likewise, $e\inv\in \suppo(n^*)$ has $s(e\inv)=r(e) = P( e \HleftX \kappa)\coloneqq  v$, so it follows again from Part~\ref{item:DGN Proposition 4.4(1)} that $\intA{e \HleftX \kappa}$ is an element of $\dom(n^*)$. 
   To see that $\alpha_{n}(\intA{\kappa})=\intA{e \HleftX \kappa}$, it suffices by the defining property~\eqref{eq:def'n of alpha_n} of $\alpha_n$ to show that for any $b\in C_c(S; \E_{S} )$, we have
    \[
        \intA{\kappa} (n^* b n ) = \intA{(e \HleftX\kappa)}(b)\, \intA{\kappa}(n^*n)
        .
    \]
    We compute
    \begin{align*}
        \intA{\kappa} (n^* b n ) 
        &=
        \sum_{   \{ \pi(\sigma)\in S(u) \}   }\overline{\kappa(\sigma)} \, (n^* b n)( \sigma)
        \\
        &=
        \sum_{   \{ \pi(\sigma)\in S(u) \}   }\overline{\kappa(\sigma)} |n(e)|^{2} b(e\sigma e\inv)
        &&\text{as $e\in \suppo(n)\cap \E u $ 
        and $n\in N$}
        \\
        &=
        \sum_{    \{\pi(\tau)\in S(v)\}    }\overline{\kappa(e\inv\tau e)} |n(e)|^{2} b(\tau)
        &&\text{by reordering}
        \\&
        =\intA{(e \HleftX\kappa)}(b)\,|n(e)|^{2}.
    \end{align*}
    As $\suppo(n^*n)\subset \cU$ and $\kappa(u)=1$, we have
    \begin{align*}
        \intA{\kappa}(n^*n)
        &=
        \sum_{   \{ \pi(\sigma)\in S(u) \}   }\overline{\kappa(\sigma)}\, (n^* n)( \sigma)
        =
        (n^* n)(u) = |n(e)|^{2},
    \end{align*}
    which finishes the claim.
\end{proof}

\begin{prop}\label{prop:Weyl explained}
   Let $\Phi\colon A\to B$ be the conditional expectation.  For $n,m\in N$, $u\in \cU $, and $\kappa\in \twPD(u)$, the following three statements are equivalent:
    \begin{enumerate}[label=\textup{(\Roman*)}]
        \item\label{it:I} $\llbracket n,\intA{\kappa}\rrbracket=\llbracket m,\intA{\kappa}\rrbracket$  (equality in the Weyl twist);
        \item\label{it:II} $\intA{\kappa}(\Phi(m^*n))> 0$;
        \item\label{it:III} There exist elements $e\in \E u$  and $\sigma\in \E_{S} (u)$ such that $n(e)>0$, $m(e\sigma)>0$, and $\kappa(\sigma)=1$.
    \end{enumerate}
    Likewise, the following three are equivalent:
    \begin{enumerate}[label=\textup{(\roman*)}]
        \item\label{it:i} $[ n,\intA{\kappa} ]=[m,\intA{\kappa}]$    (equality in the Weyl groupoid);
        \item\label{it:ii} $\intA{\kappa}(\Phi(m^*n))\neq 0$;
        \item\label{it:iii} $q_{\E}(\suppo(n)\cap \E u ) = q_{\E}(\suppo(m)\cap \E u )\neq \emptyset$.
    \end{enumerate}
\end{prop}

\begin{proof}
    If $nb=mb'$ for $b,b'\in B$, then since $\intA{\kappa}\colon B\to\mathbb{C}$ is a homomorphism and $\Phi$ is $B$-linear, it is easy to see that
    \[
        \intA{\kappa}(\Phi(m^*n))\,\intA{\kappa}(b)
        =
        \intA{\kappa}(m^*m)\,\intA{\kappa}(b').
    \]
    This is the only observation needed to prove \ref{it:I}$\implies$\ref{it:II} and  \ref{it:i}$\implies$\ref{it:ii}. For more details, see the proof of \cite[Lemma 4.3]{DGN:Weyl} in the case that the twist is induced by a $2$-cocycle.

    For \ref{it:II}$\implies$\ref{it:III}, note that $m^*n\in N$, so that $\Phi(m^*n)=(m^*n)|_{ \E_{S} }\in C_c(S;\E_{S})$ and hence by~\eqref{eq:phi_kappa},
    \[
    \intA{\kappa} (\Phi(m^*n)) = \sum_{\{\pi(\sigma)\in S(u)\}} \overline{\kappa(\sigma)} \, (m^*n)(\sigma)
    .
    \]
    Since we assumed $\intA{\kappa} (\Phi(m^*n))>0$ and since $\suppoG(m^*n)$ is a bisection, there thus exists a unique $s\in S(u)$ such that $\overline{\kappa(\sigma)}\, (m^*n)(\sigma) > 0$ for any $\sigma\in\pi\inv(\{s\})$. So if $\tau$ is the unique lift of $s$ such that $\kappa(\tau)=1$ (see Remark~\ref{rmk:Weyl twist}), then
    $\intA{\kappa} (\Phi(m^*n))>0$ implies that $(m^*n)(\tau)>0$.
    
    For this $\tau$, we have
    \[
    0< (m^*n)(\tau) = \sum_{\{\pi(f)\in Gu\} } m^*(\tau f\inv)\, n(f)
    =
    \sum_{\{\pi(f)\in Gu\} } \overline{m( f \tau\inv)} \, n(f).
    \]
    Since $\suppoG(n)$ is a bisection, there exists a unique $g\in Gu $ such that $\overline{m( f \tau\inv)} n(f)>0$ for all $f\in \pi\inv(\{g\})$.
    Using $\mathbb{T}$-equivariance of $n$, the same trick as in Remark~\ref{rmk:Weyl twist} shows that the element $e\coloneqq  \overline{n(f)}/|n(f)|\cdot f$ is the unique element of $\E u$  with $n(e)>0$. Thus, $ (m^*n)(\tau)>0$ implies $\overline{m( e \tau\inv)}>0 $. We can now let $\sigma\coloneqq  \tau\inv$. This finishes the proof of \ref{it:II}$\implies$\ref{it:III}.

   For \ref{it:ii}$\implies$\ref{it:iii},  the above argument goes through with every ``$>0$'' replaced by ``$\neq 0$''; in particular, \ref{it:ii} implies that there exist $e\in \E u$  and $\sigma\in \E_{S} (u)$ with $n(e)\neq 0$ and $m(e\sigma)\neq 0$. Since $\suppoG(n)$ and $\suppoG(m)$ intersect $Gu $ in at most one point, this shows that 
    \[
        \suppo(n)\cap \E u  = \mathbb{T}\cdot e
        \quad\text{ and }\quad
        \suppo(m)\cap \E u  = \mathbb{T}\cdot e\sigma.
    \]
    This means exactly that  $q_{\E}(\suppo(n)\cap \E u ) = q_{\E}(\suppo(m)\cap \E u )$ and that this set is not empty, i.e., that \ref{it:iii} holds.

    To show \ref{it:iii}$\implies$\ref{it:i}, we follow the ideas in the proof of \cite[Proposition 4.5]{DGN:Weyl}.
    Let $e\in \suppo(n)\cap \E u$ . By  \ref{it:iii}, there exists $\tau\in \E_{S} (u)$ with $e\tau\inv\in \suppo(m)$. Since $m,n\in N$, we deduce
    \[
    m^*n(\tau) = \overline{m(e\tau\inv)}n(e) \neq 0.
    \]
    Since $\tau\in \E_{S} $ and since $ \E_{S} $ is open in $\E$, there exists an open neighborhood $V$ of $\tau$ in $ \E_{S} $ with $m^*n (\tau')\neq 0$ for all $\tau'\in V$. This implies that $q_{\E}(\suppo(m)\cap \E v ) = q_{\E}(\suppo(n)\cap \E v )\neq\emptyset$ for all $v\in r(V)$.

     Now fix one such $v\in r(V)$. For $f\in \suppo(n)\cap \E v $, the above asserts that there exists $\sigma\in  \E_{S} (v)$ such that $f\sigma \in \suppo(m)\cap \E v $.
    Since \(f
        \HleftX \mu =
        (f\sigma)
        \HleftX \mu\)  
    for all $\mu\in \twPD(v)$,
    Lemma~\ref{lem:DGN Proposition 4.4}\ref{item:DGN Proposition 4.4(2)} explains the equations marked with $(\dagger)$ in the following:
    \[
        \alpha_{n}(\intA{\mu}) \overset{(\dagger)}{=} \intA{
        (f
        \HleftX \mu)}=\intA{
        (f\sigma)
        \HleftX \mu}
        \overset{(\dagger)}{=} \alpha_{m}(\intA{\mu}).
    \]
    Since $\intA{\mu}$ is an arbitrary element of the image $U$ of \(P\inv \left( r(V)\right)\) under $\sharp$, this proves that $\alpha_n$ and $\alpha_m$ coincide on $U$. Since~$G$ is \etale, $r(V)$ is open in $\cU =G\z $. Since $\sharp$ is a homeomorphism (Proposition~\ref{prop:spec B}) and since $P$ is continuous (Definition~\ref{def:twPD}), we deduce that \(
        U\) is an open neighborhood of $\intA{\kappa}$ in $\widehat{B}$.  By definition of the equivalence relation that gives rise to the Weyl groupoid, this means that $[n,\intA{\kappa}]=[ m,\intA{\kappa}]$, as claimed.

    It remains to show \ref{it:III}$\implies$\ref{it:I}, so assume that there exist elements $e\in \E u$  and $\sigma\in \E_{S} (u)$ such that $n(e)>0$, $m(e\sigma)>0$, and $\kappa(\sigma)=1$; we must prove that $\llbracket n,\intA{\kappa}\rrbracket= \llbracket m,\intA{\kappa}\rrbracket$. As explained in the proof of \ref{it:ii}$\implies$\ref{it:iii}, our assumption \ref{it:III} is stronger than \ref{it:iii}, so we may invoke \ref{it:iii}$\implies$\ref{it:i} to conclude $[ n,\intA{\kappa}]= [m,\intA{\kappa}]$. This implies $\llbracket n,\intA{\kappa}\rrbracket= \llbracket zm,\intA{\kappa}\rrbracket$ for a unique $z\in\mathbb{T}$. By \ref{it:I}$\implies$\ref{it:III} applied to this equality in the Weyl twist, there exist elements $f\in \E u$  and $\tau\in \E_{S} (u)$ such that $n(f)>0$, $(zm)(f\tau)>0$, and $\kappa(\tau)=1$. Since $\suppoG(n)$ is a bisection and since $n$ is $\mathbb{T}$-equivariant, it follows from $n(e)>0$ and $n(f)>0$ that $f=e$. Since $\suppoG(zm)=\suppoG(m)$ is likewise a bisection, it follows from $m(z\cdot e\tau)=(zm)(f\tau)>0$  and $m(e\sigma)>0$ that $z\cdot \tau=\sigma$. Thus, $\mathbb{T}$-equivariance of $\kappa$ implies that $z = \overline{\kappa(\tau)}\kappa(\sigma)$, which equals $1$ by the assumption that $\kappa(\sigma)=1=\kappa(\tau)$. This proves that $\llbracket n, \intA{\kappa}\rrbracket=\llbracket zm,\intA{\kappa}\rrbracket$ equals $\llbracket m,\intA{\kappa}\rrbracket$, as needed.
\end{proof}

We need one last lemma.

\begin{lemma}\label{lem:q by cK is open map}
    The quotient map $\actiongpd{\twPD}{\E} \to \Weyltwist = \bigl(\actiongpd{\twPD}{\E}\bigr) /\cK $ is open.
\end{lemma}
\begin{proof}
    Denote the quotient map by $q$. Let $X\subset \actiongpd{\twPD}{\E}$ be closed; it suffices to prove  that
    \[Y\coloneqq  \left\{\xi\in \Weyltwist: q\inv(\xi)\subset X\right\}\] is closed.
    By definition of the quotient topology, this claim is equivalent to $q\inv (Y)$ being closed in $\actiongpd{\twPD}{\E}$. Using the description of cosets $(e,\kappa)\cK \in \Weyltwist$ in Equation~\eqref{eq:cosets of cK},
    we see that 
    \begin{align}
        \label{eq:q-inv (Y)}
    q\inv (Y)
    =
    \big\{
        (e,\kappa) \in \actiongpd{\twPD}{\E} : \forall \tau\in \E_{S}(s(e)), (\overline{\kappa(\tau)}\cdot e\tau,\kappa)\in X
    \big\}
    .
    \end{align}
    Now suppose that $(e_i,\kappa_i)_i$ is a net in $q\inv (Y)$ that converges to $(e,\kappa)$ in $\actiongpd{\twPD}{\E}$. To show that $(e,\kappa)\in q\inv(Y)$, fix $\tau\in \E_{S}(s(e))$; we must check that $(\overline{\kappa(\tau)}\cdot e\tau,\kappa)\in X$.
    
    As $\kappa_i\to\kappa$, we have $u_i\coloneqq  P(\kappa_i)\to P(\kappa)=: u$. Since the source map of $\E_{S}$ is open, we can lift (a subnet of) $(u_i)_i$ to a net $\tau_i\in \E_{S}(u_i)$ with $\tau_i\to \tau$; in particular, $\kappa_i(\tau_i)\to\kappa(\tau)$ by definition of the topology on $\twPD$. Since  $(e_i,\kappa_i)\in q\inv (Y)$, we see from Equation~\eqref{eq:q-inv (Y)} that $(\overline{\kappa_i(\tau_i)}\cdot e_i\tau_i,\kappa_i)\in X$. As $X$ is closed, we deduce that the limit $(\overline{\kappa(\tau)}\cdot e\tau,\kappa)$ of this net is also in $X$,
    as needed.
\end{proof}

We are now ready to prove our description of the Weyl twist  of the non-traditional Cartan pair $B=\csr(S; \E_{S} ) \subset A=\csr(G;\E)$:

\begin{proof}[Proof of Theorem~\ref{thm:Weyl twist}]
    Fix $(e,\kappa)\in \actiongpd{\twPD}{\E} $ and let $u\coloneqq  s(e)=P(\kappa)$. Since we chose $n$ such that $n(e)>0$,  it follows from Lemma~\ref{lem:DGN Proposition 4.4}\ref{item:DGN Proposition 4.4(1)} that $\intA{\kappa}\in\dom(n)$, so $\llbracket n, \intA{\kappa} \rrbracket$ is indeed a well-defined element of $\Sigma_{B\subset A}$. If $m\in N$ is another element with $m(e)>0$, then since $\kappa(u)=1$, it follows from Proposition~\ref{prop:Weyl explained}, \ref{it:III}$\implies$\ref{it:I}, that $\llbracket 
        n, \intA{\kappa}
     \rrbracket=\llbracket 
        m, \intA{\kappa}
     \rrbracket$, which proves that $\tilde{\Delta}(e,\kappa)$ does not depend on the choice of $n$.

    \smallskip
    
   The proof that $\tilde{\Delta}$ is a homomorphism is similar to the proof of \cite[Lemma 4.7]{DGN:Weyl}: Suppose that $( f ,\mu)$ and $( e ,\kappa)$ are composable in $\actiongpd{\twPD}{\E} $, that is, $\mu= e \HleftX \kappa$, so that
   \[
   ( f , e \HleftX \kappa)( e ,\kappa)
    =
    ( fe ,\kappa).
   \]
   Now suppose that $n,m\in N$ are such that $n(e),m(f)>0$, and thus
   \begin{align*}
         \tilde{\Delta} ( f , e \HleftX \kappa) = 
         \llbracket
            m, \intA{e \HleftX \kappa}
        \rrbracket
        \quad\text{ and }\quad
         \tilde{\Delta} ( e, \kappa) = 
         \llbracket
            n, \intA{\kappa}
        \rrbracket
    \end{align*}
    and hence, by definition of the groupoid structure on $\Sigma_{B\subset A}$,
    \begin{align*}
        \tilde{\Delta} ( f , e \HleftX \kappa)\tilde{\Delta} ( e ,\kappa)
        =
        \llbracket
        mn,\intA{\kappa}
        \rrbracket.
   \end{align*}
   Thus, to see that $\tilde{\Delta} ( f , e \HleftX \kappa)\tilde{\Delta} ( e ,\kappa)=\tilde{\Delta} ( fe ,\kappa)$, it suffices to note that $mn\in N$ and that $mn(fe)=m(f)n(e)>0$.

    \smallskip

   As explained earlier,  $\tilde{\Delta}$ is surjective since every element of the Weyl twist can be written as $\llbracket n,x\rrbracket$ for some $n\in N$ and some $x\in \dom(n)\subset \widehat{B}$ (\cite[Proposition 4.1]{DGNRW:Cartan}, \cite{BG:2022:DGNRW}, \cite[Lemma 6.13.]{DL:ZStwist-pp}) and since every element of $\widehat{B}$ is of the form $\intA{\kappa}$ (Proposition~\ref{prop:spec B}).

    \smallskip

   The proof that $\tilde{\Delta}$ is  continuous and open is similar to the proof of \cite[Lemma 4.9]{DGN:Weyl}: First, for continuity, suppose that $( e_{i},\kappa_{i})_{i}$ is a net in $\actiongpd{\twPD}{\E}$ that converges to $( e ,\kappa)$. Let $n\in N$ be such that $n(e)>0$, so that $\tilde{\Delta}( e ,\kappa)=\llbracket 
        n, \intA{\kappa}
    \rrbracket $. Since $\suppo(n)$ is open, there exists $i_0$ such that, for all $i\geq i_0$, we have $e_{i}\in \suppo(n)$; if we let $z_{i}\in \mathbb{T}$ be the unique element such that $(z_{i}n)(e_{i})=n(z_{i}\cdot e_{i})>0$, then 
   \[
    \tilde{\Delta} ( e_{i},\kappa_{i}) = 
    \llbracket 
        z_{i} n, \intA{\kappa_{i}}
    \rrbracket 
    \]
    by definition of $\tilde{\Delta}$.  Since $\mathbb{T}$ is compact, and since it suffices to prove that \emph{a subnet} of $(\tilde{\Delta}(e_{i},\kappa_{i}))_{i}$ converges to $\tilde{\Delta}(e,\kappa)$, we can without loss of generality assume that the net $(z_{i})_{i}$ converges to, say, $z\in\mathbb{T}$. Since $n(z_{i}\cdot e_{i})>0$ for all $i$, continuity of $n$ implies that $z n(e)=n(z\cdot e)>0$. Since $n(e)>0$ by assumption, we conclude that the limit $z$ of $(z_{i})_{i}$ must be equal to $1$, meaning that $z_{i}n \to n$.
    Since $\kappa_{i}\to\kappa$ and since $\sharp$ is a homeomorphism, we have $\intA{\kappa_{i}}\to\intA{\kappa}$. The definition of the topology on $\Sigma_{B\subset A}$ then shows that
    \[
    \tilde{\Delta} ( e_{i},\kappa_{i}) = 
    \llbracket 
        z_{i} n, \intA{\kappa_{i}}
    \rrbracket 
    \to
    \llbracket 
        n, \intA{\kappa}
    \rrbracket 
    =    
    \tilde{\Delta} ( e ,\kappa).
    \]
    This proves that $\tilde{\Delta}$ is continuous.

    \smallskip

    To see that $\tilde{\Delta}$ is open, assume that $(\xi_{i})_{i}$ is a net in $\Sigma_{B\subset A}$ that converges to an element $\tilde\Delta(e,\kappa)$; we must show that there exists a subnet $(\xi_{f(j)})_{j}$ and net $(\eta_{j})_j$ in $\actiongpd{\twPD}{\E} $ that converges to $(e,\kappa)$ and   satisfies $\tilde\Delta(\eta_{j})=\xi_{f(j)}$.

    Let $u_{i}\coloneqq  s(\xi_{i})$ and $u\coloneqq  s(e)=P(\kappa)$, and fix  $n\in N$ with $n(e)>0$, so that $\tilde\Delta(e,\kappa)=\llbracket n, \intA{\kappa}\rrbracket$.  Since $\sharp$ is a homeomorphism, convergence $\xi_{i}\to\xi$ in $\Sigma_{B\subset A}$ implies that, for any open neighborhood $U\subset \sharp\inv(\dom(n))$ of $\kappa$ and $V\subset\mathbb{T}$ of $1$, there exists $i_{U,V}$ such that for all $i\geq i_{U,V}$,
    \[
        \xi_{i}
        =
            \llbracket 
                z_{i} n, \intA{\kappa_{i}}
            \rrbracket 
            \quad\text{ for unique }
            z_{i}\in V\quad\text{ and }\quad \kappa_{i}\in U
        .
    \]
    Going forward, we will only consider $i\geq i_{\sharp\inv(\dom(n)),\mathbb{T}}$, and we let $z_i\in\mathbb{T}$ and $\kappa_i\in \sharp\inv(\dom(n))$ be as above.
     As $\intA{\kappa_{i}}\in\dom(n)$, we have 
    $\suppo(n)\cap \E u_{i} \neq\emptyset$ by Proposition~\ref{lem:DGN Proposition 4.4}\ref{item:DGN Proposition 4.4(1)},  so 
    we may let $e_{i}$ be the unique element of $\E u_{i} $ such that $z_{i} n (e_{i})>0$. By choice of $e_{i}$, we have 
    \[
        \tilde{\Delta}(e_{i},\kappa_{i}) = \xi_{i}
        \quad\text{ and }\quad
        \tilde\Delta(e,\kappa) = \xi,
    \]
    so it suffices to show that a subnet  of $(e_{i},\kappa_{i})_{i}$ converges to $(e,\kappa)$.
    
    Firstly, since $z_{i}\in V$ and $\kappa_{i}\in U$ for $i\geq i_{U,V}$, we see that $z_{i}\to 1$ and $\kappa_{i}\to \kappa$. As $s(e_{i})=u_{i}=s(\xi_i)=P(\kappa_{i})$, the latter convergence implies that $s(e_{i})\to P(\kappa)=u=s(e)$. Since $\suppoG(n)$ is a bisection that contains $\pi(e_{i})$ and $\pi(e)$, we deduce that $\pi(e_{i})\to \pi(e)$. Since $\pi$ is open and since we may pass to subnets, we can without loss of generality assume that there exists $f_{i}\in \E$ with $\pi(f_{i})=\pi(e_{i})$ and with $f_{i}\to e$. The former implies that $e_{i}=w_{i}\cdot f_{i}$ for some $w_{i}\in\mathbb{T}$.
    Since $\mathbb{T}$ is compact, we can (again by passing to a subnet) without loss of generality assume that $w_{i}\to w$ for some $w\in\mathbb{T}$. By continuity of $n$ and since $z_i\to 1$, we get
    \[
        z_{i} n(e_{i}) = n(z_{i}\cdot e_{i}) = n(z_{i}w_{i} \cdot f_{i}) \to n(w\cdot e)=w n(e).
    \]
    As $z_{i} n(e_{i})>0$ by choice of $e_{i}$, it follows that $w n(e)>0$. This implies $w=1$ because $n(e)>0$. Therefore, $e_{i} = w_{i}\cdot f_{i} \to w\cdot e = e$, finishing our proof that $\tilde{\Delta}$ is open.

    \smallskip

    Next, we compute the kernel of $\tilde{\Delta}$. By definition of the Weyl twist, an element $\llbracket n, \intA{\kappa}\rrbracket$ is in the unit space if and only if there exist $b,b'\in B$ such that $\intA{\kappa}(b)$ and $\intA{\kappa}(b')$ are positive and $nb=b'$. Since the conditional expectation $\Phi$ fixes $B$ and is $B$-linear, and since $\intA{\kappa}\colon B\to\mathbb{C}$ is a homomorphism, it follows that $b'=\Phi(b')=\Phi(nb)=\Phi(n)b$ and hence
    \begin{align*}
        \intA{\kappa}(b') = \intA{\kappa}(\Phi(n)b) = \intA{\kappa}(\Phi(n))\intA{\kappa}(b).
    \end{align*}
    As both $\intA{\kappa}(b')$ and $\intA{\kappa}(b)$ are positive, we conclude that $\intA{\kappa}(\Phi(n))$ is positive. As $\Phi(n)=n|_{ \E_{S} } \in C_c(S; \E_{S} )$, we may use the formula at~\eqref{eq:phi_kappa}: if $u\coloneqq  P(\kappa)$, then
    \begin{equation}\label{eq:determining n}
        \sum_{   \{ \pi(\sigma)\in S(u) \}   }\overline{\kappa(\sigma)} n( \sigma)
         \overset{\eqref{eq:phi_kappa}}{=} 
        \intA{\kappa}(\Phi(n))
        >0.
    \end{equation}
    Therefore, if an element $(e,\kappa)\in \actiongpd{\twPD}{\E} $ is mapped into the unit space of $\Sigma_{B\subset A}$ by $\tilde{\Delta}$, then any $n\in N$ with $n(e)>0$ satisfies~\eqref{eq:determining n}. Since $\suppoG(n)$ is a bisection and since $u=P(\kappa)=s(e)$,~\eqref{eq:determining n} implies that $e\in  \E_{S} $ and $\overline{\kappa(e)}n(e)>0$, i.e., $\kappa(e)=1$. Therefore, the kernel of $\tilde{\Delta}$ is contained in $\cK $.

  Conversely, let $(\sigma,\kappa)\in \cK $. Since~$S$ is clopen, it follows from \cite[Proposition 3.12]{DWZ:2025:Twist} that $\csr(S; \E_{S} )\cap N \subset C_c(S; \E_{S} )$, and if $f\in \csr(S; \E_{S} )\cap N$ is such that $f(\sigma)>0$, then 
    $\tilde{\Delta}(\sigma,\kappa)=\llbracket f, \intA{\kappa}\rrbracket$. Moreover,  since $\suppoG(f)$ is a bisection, we have
    \[
        \intA{\kappa} (f) \overset{\eqref{eq:phi_kappa}}{=}  \sum_{\{\pi(\tau)\in S(u)\}}\overline{\kappa(\tau)} f( \tau) = \overline{\kappa(\sigma)} f(\sigma) = f(\sigma)> 0,
    \]
    so that $\llbracket f, \intA{\kappa}\rrbracket$ is a unit.

    By \cite[Proposition 9.]{AMP:IsoThmsGpds}, $\cK $ is a normal subgroupoid of $\actiongpd{\twPD}{\E} $ and by  \cite[Theorem 1]{AMP:IsoThmsGpds}, $\tilde{\Delta}$ descends to the (algebraic) isomorphism $\Delta$.
    Since the quotient map $\actiongpd{\twPD}{\E} \to \Weyltwist$ is continuous, the fact that $\tilde{\Delta}$ is open implies that $\Delta$ is open. Similarly, since the quotient map is open (Lemma~\ref{lem:q by cK is open map}), continuity of $\tilde{\Delta}$ implies continuity of $\Delta$  with an application of Fell's criterion.
\end{proof}

Our description of the Weyl groupoid now follows easily:
\begin{proof}[Proof of Theorem~\ref{thm:Weyl groupoid}]
    First, let us check that $\delta$ is well defined: Fix $(\dot{e},\kappa)\in \Weylgpd $, let $e$ be an arbitrary lift of $\dot{e}$ under the quotient map $\E\to \E/\E_{S}$, and let $u\coloneqq  s(\dot{e})=P(\kappa)$. Since $n\in N$ is chosen such that $\dot{e}\in q_{\E}(\suppo(n))$, there exists $\tau\in \E_{S} (u)$ with $n(e\tau)> 0$. This implies that
    \[
    n^*n (u)
    =
    n^*n \bigl( (e\tau)\inv (e\tau)\bigr)
    =
    |n(e\tau)|^{2} \neq 0.
    \]
    Thus,
    \[
    \intA{\kappa}(n^*n)
    =
    \sum_{\{\pi(\sigma)\in S(u)\}}\overline{\kappa(\sigma)}\,(n^*n)(\sigma)
    =
    (n^*n)(u)
    \neq 0,
    \]
    which shows that $\intA{\kappa}\in\dom(n)$, so $[n, \intA{\kappa}]$ is an element of $W_{B\subset A}$.
   If $m\in N$ is another element with $\dot{e}\in q_{\E}(\suppo(m))$, then Proposition~\ref{prop:Weyl explained}, \ref{it:iii}$\implies$\ref{it:i}, ensures that $[m, \intA{\kappa}]=[n, \intA{\kappa}]$, which proves that $\delta(\dot{e},\kappa)$ does not depend on the choice of $n$.

 Let 
    $ \tilde{\Delta}\colon\actiongpd{\twPD}{\E}  \to \Sigma_{B\subset A}$ denote the continuous, open, surjective homomorphism from Theorem~\ref{thm:Weyl twist}, and let $\Pi\colon \Sigma_{B\subset A}\to W_{B\subset A}$ denote the quotient/projection map, i.e., $\Pi\llbracket n,\intA{\kappa}\rrbracket = [n,\intA{\kappa}]$. Note that $\delta$  and $\tilde{\Delta}$ are defined in such a way that
    \begin{equation}\label{eq:delta in terms of the others}
        \delta(\dot{e},\kappa) = (\Pi\circ \tilde{\Delta})(e,\kappa).
    \end{equation}
    Using Equation~\eqref{eq:delta in terms of the others}, we can deduce several properties of $\delta$ without much effort: Since $\Pi$ and $\tilde{\Delta}$ are surjective groupoid homomorphisms, so is $\delta$.
    Since $\Pi$ and $\tilde{\Delta}$ are continuous and since the quotient map $\E\to \E/ \E_{S} $ is open (Lemma~\ref{lem:quotients}\ref{item:quotient map open}), $\delta$ is continuous by an application of Fell's criterion. Conversely, since $\Pi$ and $\tilde{\Delta}$ are open and since the quotient map $\E\to \E/ \E_{S} $ is continuous, $\delta$ is open, likewise using Fell's criterion. 

    The proof that $\delta$ is injective is as done in  \cite[Lemma 4.8]{DGN:Weyl}; one only needs to exchange the reference to \cite[Lemma 3.2]{DGN:Weyl} by  one to Proposition~\ref{prop:spec B}, and to \cite[Proposition 4.5]{DGN:Weyl} by Proposition~\ref{prop:Weyl explained}, \ref{it:i}$\implies$\ref{it:iii}.
\end{proof}

Theorem~\ref{thm:A} now follows as a combination of Theorems~\ref{thm:Weyl groupoid} and~\ref{thm:Weyl twist}.

\subsection{Corollaries and applications of  Theorems~\ref{thm:Weyl groupoid} and~\ref{thm:Weyl twist}}\label{ssec:corollaries:A}

\begin{cor}\label{cor:effective and top princ}
    Suppose
        $\E$ is a twist over a \LCH, \etale\ groupoid~$G$; $S$ is a clopen normal subgroupoid of~$G$ with $S\z=G\z$; and~$S$ satisfies the conditions~\eqref{eq:thm:DWZ:maximal}  and~\eqref{eq:thm:DWZ:ricc} of Theorem~\ref{thm:DWZ}.     
    Then the
    groupoid 
    $\Weylgpd $ is effective. If~$G$ is second countable, then the action of $G/S$ on $\twPD$ described in Corollary~\ref{cor:G/S:HleftX is an action} is topologically free.
\end{cor}

\begin{proof}
    By Theorem~\ref{thm:DWZ}, $\csr(S;\E_{S})$ is a Cartan subalgebra of $\csr(G;\E)$. Since Renault's Weyl groupoid of a Cartan pair is effective \cite[Theorem 1.2]{Raad:2022:Renault}, Theorem~\ref{thm:Weyl groupoid} implies that $\Weylgpd $ is effective. If~$G$ is second countable, then by \cite[Proposition 3.6]{Renault:Cartan}, this means that $\Weylgpd $ is topologically principal which means that the action of $G/S$ on $\twPD$ is topologically free (Remark~\ref{rmk:free=principal etc}).
\end{proof}

\begin{corollary}[{cf.\ \cite[Proposition 3.1]{IKRSW:2021:Extensions}, \cite[Lemma 5.1]{Renault:2023:Ext}}]\label{cor:Cartan pair iso}
    In the setting of Assumption~\ref{as:E,G,S,v2} with  $B=\csr(S; \E_{S} )\subset A=\csr(G;\E)$ the associated non-traditional Cartan pair, the following diagram commutes:
    \begin{equation}\label{diag:twist}
        \begin{tikzcd}[row sep=tiny, 
        column sep = large]
        (\overline{z},\intA{\kappa})\ar[r, phantom, "\in"]&[-35]\mathbb{T}\times \widehat{B} \arrow[r,hook, "\iota"] &\Sigma_{B\subset A}  \arrow[r,two heads, "\Pi"]&
            W_{B\subset A}
        \\[20pt]
        &\mathbb{T}\times \twPD \arrow[r,hook]
        \ar[u] &\Weyltwist  \arrow[r,two heads]\ar[u, "\Delta"]&
            \Weylgpd \ar[u, "\delta"]
        \\
        (z,\kappa)\ar[rr,mapsto]\ar[uu,mapsto]\ar[ru, phantom, "{\text{\rotatebox[origin=c]{45}{$\in$}}}"]&  & \bigl( z\cdot P(\kappa),\kappa\bigr)\cK 
        \\
        && (e,\kappa)\cK \ar[r,mapsto] &(\dot{e},\kappa)
        \end{tikzcd}
    \end{equation}
    In other words, $\Weyltwist$ is   a twist over the groupoid $\Weylgpd $ that is isomorphic to the Weyl twist $\Sigma_{B\subset A}\to W_{B\subset A}$.  In particular, there 
     exists a *-isomorphism
    \begin{equation}\label{eq:iso between reduced}
        \csr(G;\E)
        \overset{\cong}{\longrightarrow}
        \csr
        \bigl(
            \Weylgpd  ;
            \Weyltwist 
        \bigr)
    \end{equation}
    that maps the Cartan subalgebra $\csr(S; \E_{S} )$ onto $C_0(\twPD)$.
\end{corollary}

\begin{proof}
    The right-hand square commutes by definition of $\delta$; see also Equation~\eqref{eq:delta in terms of the others}. For the left-hand square, fix $(z,\kappa)\in \mathbb{T}\times\twPD$.
    Since $u\coloneqq  P(\kappa)\in \cU\subset \E_{S}$, we can find an element $n\in N\cap B$ such that $n(z\cdot u)=1$, so that
    \[
        \tilde\Delta(z\cdot u,\kappa)
        =
        \llbracket
            n,\intA{\kappa}
        \rrbracket
    \]
    by definition of $\tilde{\Delta}$.
    Then
    \[
        \intA{\kappa}(n)
        =
        \sum_{\{\pi(\sigma)\in S(u)\}}\overline{\kappa(\sigma)}n(\sigma)
        =
        \overline{\kappa(z\cdot u)}\,n(z\cdot u)
        =
        \overline{\kappa(z\cdot u)}
        =
        \overline{z \kappa(u)}
        =
        \overline{z}.
    \]
    Since $n\in B$, the definition of $\iota\colon \mathbb{T}\times\widehat{B}\to\Sigma_{B\subset A}$ in~\eqref{eq:def:iota}
    implies that
    \(
    \iota(\overline{z},\intA{\kappa})
    =
        \llbracket
            n,\intA{\kappa}
        \rrbracket
        ,
    \) as claimed.
    
    Since $\Sigma_{B\subset A}$ is known to be a twist over $W_{B\subset A}$, Theorems~\ref{thm:Weyl twist} and~\ref{thm:Weyl groupoid} 
    and commutativity of the diagram
    immediately imply that
    $\Weyltwist \to \Weylgpd $ is a twist that is isomorphic to $\Sigma_{B\subset A}\to W_{B\subset A}$. The existence of the Cartan-preserving isomorphism is proved exactly as \cite[Corollary 5.5]{DGN:Weyl}.
\end{proof}

\begin{remark}
    Borrowing words of Jean Renault (\cite[p.\ 265]{Renault:2023:Ext}), 
    Corollary~\ref{cor:Cartan pair iso} can be viewed as a ``partial Gelfand transform.'' By modding out the maximal open abelian subgroupoid~$S$, we have turned the non-effective~$G$ into an effective groupoid, and the only remaining abelian part is  now just the new groupoid's unit space, $\twPD$. 
\end{remark}
 
\begin{remark}\label{rmk:effective}
    If one assumes that $\Weylgpd$ is second countable and if, somehow, one already knows that it is effective, then the existence of the Cartan-preserving isomorphism of reduced $\cs$-algebras at~\eqref{eq:iso between reduced} follows from \cite[Corollary 3.11]{IKRSW:2021:Pushout}, and so \cite[Theorem 5.9]{Renault:Cartan} implies that 
    $\Weylgpd $ must be the Weyl groupoid and $\Weyltwist$ the Weyl twist. In other words, assuming second countability, Corollary~\ref{cor:effective and top princ} allows one to deduce Corollary~\ref{cor:Cartan pair iso} from the existing literature without invoking Theorem~\ref{thm:Weyl twist}.
    Nevertheless, we believe that the explicit description of the groupoid isomorphisms in Theorems~\ref{thm:Weyl groupoid} and~\ref{thm:Weyl twist} can be valuable for computations even in the second countable case.
\end{remark}
 
\begin{remark}
If $G$ is second countable, then the isomorphism at~\eqref{eq:iso between reduced} is known to exist in greater generality (also for the full $\cs$-algebras); see \cite[Theorem 5.5]{Renault:2023:Ext}, where~$G$ is not assumed to be Hausdorff or \etale, and~$S$ is not maximal in any sense and is merely assumed to be locally closed (rather than clopen).
\end{remark}

\begin{example}[continuation of Examples~\ref{ex:c_theta:setup} and~\ref{ex:c_theta:S=Zx0}]\label{ex:c_theta:Weyl groupoid}
    We return to the rotation algebras $A_{\theta}$ from the view point of a twisted group algebra, that is, $A_{\theta}=\csr(\mathbb{Z}^{2},\cocycle_{\theta})$ where  $\cocycle_{\theta}$ is the $2$-cocycle as at~\eqref{eq:c_theta} on $G=\mathbb{Z}^{2}$; as before, let $\lambda\coloneqq  \mathsf{e}^{2\pi i \theta}$. We will use our results to compute the Weyl groupoid of all Cartan pairs that come from subgroups of $\mathbb{Z}^{2}$. (We will hold off on determining the Weyl twist until Example~\ref{ex:c_theta:trivial Weyl twist}.)
    
    As before, let $\E_{\theta}\coloneqq  \mathbb{T}\times_{\cocycle_{\theta}}G$ be the twist associated to $\cocycle_{\theta}$, and let \[S=\mathbb{Z}\cdot (k,0)\oplus \mathbb{Z}\cdot (m,n)\]
    be a subgroup of~$G$, where $k,m,n\geq 0$; write $\E_{\theta,S}\coloneqq  (\E_{\theta})_{S} = \mathbb{T}\times_{\cocycle_{\theta}|_{S\times S}}S$.  \cite[Example 4.8]{DWZ:2025:Twist} lists conditions on the non-negative integers $k,m,n$ that are necessary and sufficient for~$S$ to give rise to a Cartan subalgebra in $\csr(G,\cocycle_{\theta})$. While these conditions depend on whether or not $\theta$ is rational, they all imply that the $2$-cocycle is given by the following formula when evaluated at elements of~$S$:
    \[
        \cocycle_{\theta}
        \big(l(k,0)+r(m,n),l'(k,0)+r'(m,n)\big)=
        \lambda^{r r' mn}
        .
    \]
    In particular, if we let
    \[
        \coboundary_{\theta}\colon S\to \mathbb{T}, 
        \quad
        \coboundary_{\theta}\big( l(k,0)+r(m,n) \big) \coloneqq   \mathsf{e}^{-\pi i \theta r^{2} mn}
        ,
    \]
    then 
    \[\cocycle_{\theta}( \vec{s} , \vec{t} )\coboundary_{\theta}( \vec{s} + \vec{t} )=\coboundary_{\theta}( \vec{s} )\coboundary_{\theta}( \vec{t} )\]
    for $ \vec{s} , \vec{t} \in S$, proving that $\cocycle_{\theta}|_{S\times S}$ is a coboundary. As mentioned in Example~\ref{ex:coboundary}, this means that the map
    \[
        \mfs\colon S\to \E_{\theta,S}
        \quad\text{given by}\quad
        s\mapsto \bigl( \overline{\coboundary_\theta(s)},s\bigr)
    \]
    is a homomorphic section of $\pi$.
    Therefore, we can identify the twisted dual with the usual Pontryagin dual as explained in Remark~\ref{rmk:twPD for no twist}: an element $\kappa\in\twPD[\E_{\theta}]$ corresponds to the element $\underline{\kappa}\colon S\to\mathbb{T}$ of $\widehat{S}$ defined by
    \begin{equation}\label{eq:underline-kappa}
        \underline{\kappa}(\vec{s}) = \overline{\coboundary_{\theta}(\vec{s})}\kappa(1;\vec{s}),
        \text{ and conversely }
        \kappa(z;\vec{s}) = z \coboundary_{\theta}(\vec{s})\underline{\kappa}(\vec{s}).
    \end{equation}
    
    Using Equation~\eqref{eq:c_theta:conjugation}, one readily verifies that the action of $G$ on $\twPD[\E_{\theta}]$ (Hooptedoodle~\ref{hoop:G also acts on twPD}) translates to the following~$G$-action on $\widehat{S}$:
    \begin{equation}\label{eq:c_theta:HleftX}
        (g_{1},g_{2})\HleftX \underline{\kappa}\colon (s_{1},s_{2})
        \mapsto
        \lambda^{g_{1} s_{2} - g_{2} s_{1}}
        \underline{\kappa}(s_{1},s_{2})
        \quad
        \text{ for } (s_{1},s_{2}) \in S\subset \mathbb{Z}^{2}.
    \end{equation}

    \myparagraph{Case 1: $\theta\notin\mathbb{Q}$} According to \cite[Example 4.8]{DWZ:2025:Twist},~$S$ must be of the form $\mathbb{Z}\cdot (m,n)$ with $\gcd(m,n)=1$. We claim that $\Weylgpd[\widehat{S}]$ is canonically isomorphic to the action groupoid $\actiongpd[\theta]{\mathbb{T}}{\mathbb{Z}}$ from Example~\ref{ex:rotation as transformation}, similarly to what we have already seen in  Example~\ref{ex:c_theta:S=Zx0} for the case $(m,n)=(0,1)$.  Indeed, if we identify $\widehat{S}$ with $\mathbb{T}$ in the obvious way, then the action of $G/S$ becomes
    \begin{equation}\label{eq:c_theta:HleftX for mathbbT}
        (g_{1},g_{2})+S\HleftX w
        \coloneqq  
        \lambda^{g_{1} n - g_{2} m}
        w
        .
    \end{equation}
    Since $\gcd(m,n)=1$, the quotient $G/S$ is isomorphic to $\mathbb{Z}$ via $ (g_{1},g_{2})+S\mapsto g_{1}n - g_{2}m$, and the $G/S$-action on $\widehat{S}$ becomes the $\mathbb{Z}$-action on $\mathbb{T}$ given by
    \(
        d\HleftX w
        =
        \lambda^{d}
        w
    \) 
    for $d\in\mathbb{Z}$, $w\in\mathbb{T}$.
    Using Theorem~\ref{thm:Weyl groupoid}, this proves that the Weyl groupoid of each of the Cartan pairs
    \(
        \csr(\mathbb{Z}\cdot (m,n),\cocycle_{\theta})\subset \csr(\mathbb{Z}^{2},\cocycle_{\theta})
    \) is (isomorphic to) the action groupoid
    $\actiongpd[\theta]{\mathbb{T}}{\mathbb{Z}}$.

    \myparagraph{Case 2: $\theta\in\mathbb{Q}$} Since the case $\theta=0$ is trivial, let us assume $\theta=p/q$ for  $p\in\mathbb{Z}^{\times}$, $q\in \mathbb{N}$, and $\gcd(p,q)=1$. To make our lives easier, let us first argue that we can without loss of generality assume that the integer $m$ in the definition of~$S$ can be chosen to be $0$:
    
    The Smith Normal form of the matrix $\cM\coloneqq  \left[\begin{smallmatrix}k & m \\ 0 & n\end{smallmatrix}\right]$ is given by 
    \[
    \cN\coloneqq  \left[\begin{smallmatrix}k' & 0 \\ 0 & n'\end{smallmatrix}\right],
    \quad\text{ where }\quad k'\coloneqq  \gcd(k,m,n)\text{ and }n'\coloneqq  nk/k'.
    \]
    In particular, there exist $\cP,\cQ\in\mathrm{GL}_{2}(\mathbb{Z})$ such that $\cP\cM\cQ=\cN$. Note that $S=\cM\mathbb{Z}^{2}$ (once we write $\vec{g}\in \mathbb{Z}^{2}=G$ as column- rather than row-vectors) and that, if $T\coloneqq  \cN\mathbb{Z}^{2} \leq G$, then
    \[
    \begin{tikzcd}[row sep = tiny]
        G/T \ar[r]& G/S&\quad\text{ and }\quad&
        S\ar[r]&T
        \\
        \vec{g}+T\ar[r,mapsto]& \det(\cP)\cP\inv \vec{g}+S
         &&
        \vec{s}\ar[r,mapsto]& \cP \vec{s}
    \end{tikzcd}
    \]
    are isomorphisms. 
    Consequently, the given $G/S$-action $\HleftX$ on $\widehat{S}$ induces a (new) action $\XleftG$ of $G/T$ on $\widehat{T}$ if we define for $\vec{g}\in G$ and $\mu \in\widehat{T}$,
    \[
        \vec{g}+T \XleftG \mu 
        \coloneqq  
        \bigl(
        (\det(\cP)\cP\inv \vec{g}) \HleftX (\mu \circ \cP )\bigr) \circ \cP\inv
        . 
    \]
A quick computation using~\eqref{eq:c_theta:HleftX} for $\HleftX$ reveals that for $\mu \in\widehat{T}$ and $\vec{t}=(t_{1},t_{2})$, this formula boils down to
    \begin{align*}
        (\vec{g}+T \XleftG \mu )(\vec{t})
        &=
        \lambda^{g_{1}t_{2}-g_{2}t_{1}}
        \mu (\vec{t})
        ,
    \end{align*}
    which is again just Equation~\eqref{eq:c_theta:HleftX}. In summary: the triple $(G,\cocycle_\theta, S)$ for $S=\mathbb{Z}\cdot (k,0)\oplus \mathbb{Z}\cdot (m,n)$ gives rise to a $G/S$-action $\HleftX$ given by Equation~\eqref{eq:c_theta:HleftX} on $\widehat{S}$ which, in turn, gives rise to a $G/T$-action $\XleftG$ on $\widehat{T}$ for $T=\mathbb{Z}\cdot (k',0)\oplus \mathbb{Z}\cdot (0,n')$ that is on-the-nose identical to the $G/T$-action induced from the triple $(G,\cocycle_\theta,T)$. In other words,
    the map
    \[
        \actiongpd[\HleftX]{\widehat{S}}{G/S}
        \cong
        \actiongpd[\XleftG]{\widehat{T}}{G/T}
        =
        \actiongpd[\HleftX]{\widehat{T}}{G/T}
        \quad\text{ given by }\quad
        (\vec{g}+S, \underline{\kappa})
        \mapsto (\det(\cP)\cP \vec{g}+T, \underline{\kappa}\circ \cP\inv )
    \]
    is an isomorphism of action groupoids. Let us now identify $
        \actiongpd{\widehat{T}}{G/T}$ explicitly:

    Since $T = k'\mathbb{Z} \times n'\mathbb{Z}$, its dual $\widehat{T}$ is homeomorphic to $\mathbb{T}^{2}$: if we identify an element $(w_{1},w_{2})\in\mathbb{T}^{2}$ with the (slightly unusual, but eventually convenient) homomorphism
    \begin{equation}\label{eq:odd choice}
    k'\mathbb{Z} \times n'\mathbb{Z} = T \ni (t_{1},t_{2})
    \mapsto
    w_{1}^{t_{2}/n'}w_{2}^{-t_{1}/k'},
    \end{equation}
    then the action of $G/T$ on $\widehat{T}$ becomes the following action of $\mathbb{Z}/k'\mathbb{Z}\times\mathbb{Z}/n'\mathbb{Z}$ on $\mathbb{T}^2$:
    \begin{align}\label{eq:nice formula thx2odd choice}
        (g_{1}+k'\mathbb{Z},g_{2}+n'\mathbb{Z})\HleftX (w_{1},w_{2})
        &=
        \bigl(
        \lambda^{n'g_{1}}
        w_{1},
        \lambda^{k'g_{2}}
        w_{2}
        \bigr)\notag
        \\
        &=
        \bigl(
        \mathsf{e}^{2\pi ig_{1}\cdot p n' /q}
        w_{1},
        \mathsf{e}^{2\pi ig_{2}\cdot p k'/q}
        w_{2}
        \bigr)
        &&(\lambda=\mathsf{e}^{2\pi i\theta}, \theta=p/q)
    \end{align}

 According to \cite[Example 4.8]{DWZ:2025:Twist} again, $T$ gives rise to a Cartan subalgebra if and only if $k'n'$ is equal to $q$; since $k'n'=kn$, this happens if and only if~$S$ gives rise to a Cartan subalgebra. In this case, the $2$-cocycle is trivial on $T$, and we have  $pn'/q=p/k'$ and $pk'/q=p/n'$. 
  
   In summary, we have shown that the Weyl groupoid for $\csr(\mathbb{Z}^{2},\cocycle_{p/nk})$ with respect to the Cartan subalgebra 
    \(
        \csr(\mathbb{Z}\cdot (k,0)\oplus \mathbb{Z}\cdot (m,n),\cocycle_{p/nk}) 
    \)
    is canonically isomorphic to that for the Cartan
    \(
        \csr(\mathbb{Z}\cdot (k',0)\oplus \mathbb{Z}\cdot (0,n'))
    \)
    for $k'=\gcd(k,m,n)$ and $n'=nk/k'$, and 
    since $pn'/q=p/k'$ and $pk'/q=p/n'$,
    it is given by
    \[
        \bigl(
            \actiongpd[p/k']{\mathbb{T}}{\mathbb{Z}/k'\mathbb{Z}}            
        \bigr)
        \times
        \bigl(
            \actiongpd[p/n']{\mathbb{T}}{\mathbb{Z}/n'\mathbb{Z}}
        \bigr)
        .
    \]
\end{example}

\section{Trivial Weyl twist}\label{sec:trivial Weyl twist}

This section is devoted to identifying when the Weyl twist of a non-traditional Cartan pair is trivializable. That is:
\begin{prop}\label{prop:trivial Weyl twist}
    Suppose $\E$ is a twist over a \LCH, \etale\ 
    groupoid~$G$. Let $S\subset \Iso{G}$ be an open and normal subgroupoid of~$G$ such that $B\coloneqq  \csr(S; \E_{S} )$ is a Cartan subalgebra of $A\coloneqq  \csr(G;\E)$. Then the Weyl twist $\Sigma_{B\subset A}$ is trivializable if and only if there exists a continuous homomorphism
    $F\colon \actiongpd{\twPD}{\E}\to\mathbb{T}$ satisfying 
    \begin{equation}\label{eq:F:tau}
        F(\tau,\kappa)
        =
        \kappa(\tau)
        \quad\text{ for all } \tau\in \E_{S} 
        .
    \end{equation}
    In this case,
    \[
        \csr(G;\E) \cong \csr\bigl(\Weylgpd \bigr)
        \cong
        C_0(\twPD)\rtimes_{\red} G/S .
    \]
\end{prop}
It should be pointed out that it might not be particularly easy to find such a function $F$; rather, the upshot of Proposition~\ref{prop:trivial Weyl twist} is that the characterization of when $\Sigma_{B\subset A}$ is trivializable has been phrased entirely in terms of the original data $\atgext$.

The main tool needed to prove Proposition~\ref{prop:trivial Weyl twist} is the following lemma, where
\begin{equation}\label{eq:tilde pi}
    \tilde{\pi}\colon \Weyltwist\to\Weylgpd
    \quad
    \text{denotes the map}
    \quad
    (e,\kappa)\cK\mapsto (\dot{e},\kappa)
\end{equation}
from Diagram~\eqref{diag:twist}.
\begin{lemma}\label{lem:existence of mft}
    There is a one-to-one correspondence between sections  $\mft$  of $\tilde{\pi}$ and maps $F\colon \actiongpd{\twPD}{\E}\to\mathbb{T}$ satisfying
    \begin{equation}\label{eq:F:tau,v2}
        F(e\tau,\kappa)
        = F(e,\kappa)
        \kappa(\tau)
        \quad\text{ for all } \tau\in \E_{S}.
    \end{equation}
    The relationship between the maps is given by the formula
    \begin{align}\label{eq:mft and F}
        \mathfrak{t}(\dot{e},\kappa)        
        =
        ( \overline{F(e,\kappa)}\cdot e , \kappa )\cK 
        .
    \end{align}
    Moreover, $\mft$ is homomorphic if and only if $F$ is homomorphic, and $\mft$ is continuous if and only if $F$
    is continuous.
\end{lemma}
\begin{proof}
    First, suppose that $\mft$ is a section of $\tilde{\pi}$, i.e.,
    if $(f,\mu), (e,\kappa)\in \actiongpd{\twPD}{\E}$ are such that 
        $\mft (\dot{e},\kappa)
        = (f,\mu)\cK$, then $(\dot{f},\mu)=
        (\dot{e},\kappa)$ in $\Weylgpd$.
    This means that,    
    for every $(e,\kappa)\in \E\bfp{s}{P}\twPD$, there exists $\sigma(e,\kappa)\in  \E_{S} $ such that
    \[
        \mathfrak{t}(\dot{e},\kappa)
        =
        \bigl( e \sigma(e,\kappa)\inv, \kappa \bigr)\cK .
    \]
    If we let $F(e,\kappa)\coloneqq  \kappa(\sigma(e,\kappa))$, then Equation~\eqref{eq:mft and F} holds by~\eqref{eq:cosets of cK}.
    For any $\tau\in \E_{S}(P(\kappa)) $, we have
    \begin{align*}
        ( \overline{F(e,\kappa)}\cdot e , \kappa )\cK
        &=
        \mathfrak{t}(\dot{e},\kappa)
        =
        \mathfrak{t}(\dot{e}\dot{\tau},\kappa)
        =
        ( \overline{F(e\tau,\kappa)}\cdot  e\tau, \kappa )\cK
        .
    \end{align*}
    By definition of $\cK$, this is equivalent to
    \[
        1
        =
        \kappa
        \left(
            [\overline{F(e,\kappa)}\cdot  e]\inv [\overline{F(e\tau,\kappa)}\cdot e\tau]
        \right)
        =
        F(e,\kappa)
        \overline{F(e\tau,\kappa)}
        \kappa(\tau)
    \]
    and hence proves that $F$ satisfies Equation~\eqref{eq:F:tau,v2}.
    Conversely, given $F$, then the above argument in reverse shows that Equation~\eqref{eq:mft and F} determines a well-defined map $\mft$ if and only if $F$ satisfies Equation~\eqref{eq:F:tau,v2}, and since $q_{\E}(e)=q_{\E}(z\cdot e)$ for any $z\in\mathbb{T}$, the map $\mft$ is a section.

    Now, $\mft$ is homomorphic if and only if
    \begin{align*}
        \mft (\dot{f},\dot{e}\HleftX \kappa)
        \mft (\dot{e},\kappa)
        =
        \mft (\dot{f}\dot{e},\kappa)
        &&\text{ for all }
        (f,e,\kappa)\in \E^{(2)}\bfp{s}{P}\twPD.
    \end{align*}
    Since
    \begin{align*}
        \mft (\dot{f}\dot{e},\kappa)
        &=
        \bigl(  \overline{F(fe,\kappa)}\cdot fe , \kappa \bigr)\cK
        &&\text{ and}
        \\
        \mft (\dot{f},\dot{e}\HleftX \kappa)
        \mft (\dot{e},\kappa)
        &=
        \bigl( \overline{F(f,e\HleftX\kappa) F(e,\kappa)}\cdot fe , \kappa \bigr)\cK,
    \end{align*}
    it follows that $\mft$ is homomorphic if and only if
    \begin{align*}
        1
        =
        \kappa \bigl(  [\overline{F(fe,\kappa)}\cdot fe]\inv [\overline{F(f,e\HleftX\kappa) F(e,\kappa)}\cdot fe]\bigr)
        =
        F(fe,\kappa)\overline{F(f,e\HleftX\kappa) F(e,\kappa)},
    \end{align*}
    which means exactly that $F$ is a homomorphism.

    It remains to show the claim about continuity, so first suppose that $\mft$ is continuous. Assume that $(e_{i},\kappa_{i})\to (e,\kappa)$; we need to show that a subnet of $(F(e_{i},\kappa_{i}))_i$ converges to $F(e,\kappa)$ \cite[Theorem 18.1, (2)$\implies$(1)]{Munkres}. Since $q_{\E}\colon  e\mapsto\dot{e}$ and $\mft$ are continuous, we have 
   \[        
        ( \overline{F(e_{i},\kappa_{i})}\cdot e_{i} , \kappa_{i} )\cK
        \overset{\eqref{eq:mft and F}}{=}
        \mathfrak{t}(\dot{e}_{i},\kappa_{i})
        \to 
        \mathfrak{t}(\dot{e},\kappa)
        =
       (  \overline{F(e,\kappa)}\cdot e , \kappa )\cK.
   \]
   Since the quotient map $\actiongpd{\twPD}{\E} \to (\actiongpd{\twPD}{\E} )/\cK =\Weyltwist $ is open by Lemma~\ref{lem:q by cK is open map} and surjective, 
        we may invoke Fell's criterion; as we can without loss of generality pass to subnets, we may assume that there exists a net $(f_{j},\mu_{i})_{i}$ in $\E\bfp{s}{P}\twPD$ such that $(f_{i},\mu_{i})\to  (\overline{F(e,\kappa)}\cdot e,\kappa)$ and
   \begin{align*}
   (f_{i},\mu_{i})\cK
   &=
   ( \overline{F(e_{i},\kappa_{i})}\cdot e_{i} , \kappa_{i} )\cK, 
   \quad\text{ i.e., }\quad
   \mu_{i}=\kappa_{i},f_{i}\inv e_{i} \in  \E_{S}
   \quad\text{ and }\quad
   \kappa_{i} (f_{i}\inv e_{i} )
   =
   F(e_{i},\kappa_{i})
   .
   \end{align*}
   Using that $f_{i}\to \overline{F(e,\kappa)}\cdot e$, $e_{i}\to e$, and $\kappa_{i}\to \kappa$, the topology on $\twPD$ thus is designed so that
   \[
   F(e_{i},\kappa_{i}) 
   =
   \kappa_{i} (f_{i}\inv e_{i} )
   \to
   \kappa \bigl( [\overline{F(e,\kappa)}\cdot e]\inv e \bigr)
   =
   F(e,\kappa),
   \]
   proving that $F$ is continuous.
   
   Conversely, assume that $F$ is continuous. Since $q_{\E}$ is open, we may take a convergent net $(e_{i},\kappa_{i})\to (e,\kappa)$ in $\E\bfp{s}{P}\twPD$, and we must show that a subnet of $\mft(\dot{e}_{i},\kappa_{i})$ converges to $\mft(\dot{e},\kappa)$. Our continuity assumption implies
   \(      
       ( \overline{F(e_{i},\kappa_{i})}\cdot e_{i} , \kappa_{i} )
       \to (  \overline{F(e,\kappa)}\cdot  e, \kappa )
   \) in $\actiongpd{\twPD}{\E}$,
   so that continuity of $(f,\mu)\mapsto (f,\mu)\cK$ and~\eqref{eq:mft and F} then imply  continuity of $\mft$.
\end{proof}

\begin{proof}[Proof of Proposition~\ref{prop:trivial Weyl twist}]   
    By Corollary~\ref{cor:Cartan pair iso}, the Weyl twist of the pair $\csr(S; \E_{S} )\subset\csr(G;\E)$ is (canonically isomorphic to) the twist at~\eqref{eq:tilde pi}.
    By Lemma~\ref{lem:existence of mft}, the existence of a continuous, global, homomorphic section of the latter twist is equivalent to the existence of a map $F$ as claimed in the proposition. 
    As explained in Lemma~\ref{lem:c_s}, a twist admits a continuous and homomorphic section if and only if it is trivializable. Thus, the Weyl twist is isomorphic to the trivial twist if and only if $F$ exists. 
    The isomorphism of $\cs$-algebras now follows from the fact that $\csr(G;\mathbb{T}\times G)\cong \csr(G)$ and from Equation~\eqref{eq:iso between reduced} in Corollary~\ref{cor:Cartan pair iso}.
\end{proof}

\begin{example}[continuation of Examples~\ref{ex:c_theta:setup}, \ref{ex:c_theta:S=Zx0}, and~\ref{ex:c_theta:Weyl groupoid}]\label{ex:c_theta:trivial Weyl twist}
    We want to investigate whether or not the Weyl twist associated to the Cartan subalgebras of the rotation algebras is trivializable. Since the author was not able to make progress for the rational case, we will focus here on $\theta\notin \mathbb{Q}$.

     As explained in Example~\ref{ex:c_theta:Weyl groupoid}, we must have $S=\mathbb{Z}\cdot (m,n)$ with $\gcd(m,n)=1$, so that $\twPD[\E_{\theta}]$ is homeomorphic to $\mathbb{T}$: by~\eqref{eq:underline-kappa}, $w\in \mathbb{T}$ corresponds to $\kappa_{w}\in\twPD[\E_{\theta}]$ given by
    \[
        \kappa_{w}\colon \quad \mathbb{T}\times_{\cocycle_\theta}S \ni (z;rm,rn) \mapsto z
        \mathsf{e}^{-\pi i\theta r^{2}mn}
        w^{r}
        \in \mathbb{T}.
    \]
    If we use Bezout's identity to find integers $a,b$ such that $1=am+bn$, then we may define
    \(
        F\colon \actiongpd{\twPD[\E_{\theta}]}{\E_{\theta}} \to\mathbb{T}
    \)
    by
    \[
        F(z;\vec{h};\kappa_{w})
        \coloneqq  
        z \mathsf{e}^{\pi i \theta (an h_{1}^{2} - 2am h_{1}h_{2} - bm h_{2}^{2})} w^{ah_{1}+bh_{2}}.
    \]
    One can verify that $F$ satisfies~\eqref{eq:F:tau}, i.e., 
    \[
       F(z;rm,rn;\kappa_w)
       =
       \kappa_w(z;rm,rn)
       \text{ for all }r\in\mathbb{Z}.
    \]
    and that
    \[
      F(z';\vec{g};\kappa_{
      \lambda^{h_{1}n - h_{2} m}
      w})
      F(z; \vec{h};\kappa_w) 
      =
      F(z'z
      \lambda^{g_{2}h_{1}}
      ;\vec{g}+\vec{h};\kappa_w),
    \]
    which means that $F$ is a homomorphism (see also Equation~\eqref{eq:c_theta:HleftX for mathbbT}). As $F$ is continuous, it thus follows from Proposition~\ref{prop:trivial Weyl twist} that the Weyl twist is trivializable. In particular, together with Example~\ref{ex:c_theta:Weyl groupoid}, we get for $\theta\notin\mathbb{Q}$ that
    \[
        \csr (G, \cocycle_\theta) \cong \csr (\actiongpd[\theta]{\mathbb{T}}{\mathbb{Z}}),
        \text{ mapping }
        \csr (\mathbb{Z}\cdot (m,n), \cocycle_\theta)
        \text{ to }C(\mathbb{T}),
    \] 
    and so all Cartan pairs $B\subset A_{\theta}$ arising from subgroupoids in $(\mathbb{Z}^{2},\cocycle_\theta)$ are isomorphic (as Cartan pairs).
\end{example}

In the following, let $\widehat{S}$ be the untwisted dual of $S$ as described in \cite[Proposition 3.3]{MRW:1996:CtsTrace} (see also Remark~\ref{rmk:twPD for no twist}).

\begin{cor}[of Proposition~\ref{prop:trivial Weyl twist}]
    Suppose~$G$ is a \LCH, \etale\
    groupoid. Let $S\subset \Iso{G}$ be an open, normal, abelian subgroupoid of~$G$, and let $\actiongpd{\widehat{S}}{G}$ be the action groupoid associated to
    \[
    g\HleftX \chi \colon \quad
    t\mapsto \chi (g\inv t g)
    \quad\text{ for }\quad
    (t,g,\chi )\in S\bfpsr G\bfp{s}{P}\widehat{S}
    .
    \]
    If the subalgebra $\csr(S)$ of $\csr(G)$ is a Cartan subalgebra, then the associated Weyl twist  is trivializable if and only if there exists a continuous homomorphism
    $f\colon \actiongpd{\widehat{S}}{G}\to\mathbb{T}$ satisfying 
    \begin{equation}\label{eq:f:t}
        f(t,\chi)
        =
        \chi(t)
        \quad\text{ for all } t\in S(u), \chi\in\widehat{S(u)}
        .
    \end{equation}
    In this case,
    \[
        \csr(G) \cong \csr\bigl(\Weylgpd[\widehat{S}]\bigr).
    \]
\end{cor}

\begin{proof}
    Let $\E\coloneqq  \mathbb{T}\times G$ be the trivial twist, so that $\csr(S; \E_{S} )= \csr(S)$ is a Cartan subalgebra of $\csr(G;\E)= \csr(G)$ by assumption. As explained in  Remark~\ref{rmk:twPD for no twist}, the global section $\mfs\colon G\to \E, g\mapsto (1,g)$, gives rise to a homeomorphism $\mfs^*\colon \twPD\to\widehat{S}$. Since $\mfs$ is a homomorphism, one can check  that
    \begin{align}\label{eq:mfs-star is equivariant}
            \pi(e)\HleftX\mfs^*(\kappa)
            =
            \mfs^*(e\HleftX\kappa)
            \quad\text{ and }\quad
            (\mfs^*)\inv (\pi(e)\HleftX\chi)
            =
            e\HleftX (\mfs^*)\inv(\chi)
        \end{align}
    for all $e\in\E_{u}, \kappa\in\twPD(u),\chi\in\widehat{S}(u)$.
   In particular, when passing to the actions of the quotient $G/S=\E/\E_{S}$, then $\mfs^*$ and its inverse are equivariant. This proves that the map
   \begin{equation}\label{eq:mfs:iso of gpds}
       \Weylgpd
       \to
       \Weylgpd[\widehat{S}]
       \quad\text{ given by }\quad
       (\dot{e},\kappa)
       \mapsto
       (\dot{e},\mfs^*(\kappa))
   \end{equation}
   is an isomorphism of topological groupoids.
    
    Now, from any map $F\colon \actiongpd{\twPD}{\E}\to\mathbb{T}$, we define a map $f_{F}\colon \actiongpd{\widehat{S}}{G}\to\mathbb{T}$ by
    \[
        f_{F}(g,\chi) \coloneqq   F\big(\mfs(g), (\mfs^*)\inv(\chi)\big),
    \]
    and conversely, from any map $f$, we define $F_{f}$ by
    \[
        F_{f}\big((z,g),\kappa\big) \coloneqq  
        z\, f\big( g,\mfs^*(\kappa)\big).
    \]
    Using continuity of $\mfs$ and $\mfs^*$ (respectively using~\eqref{eq:mfs-star is equivariant}), one can easily check that, if one of the maps $f,F$ is continuous (respectively homomorphic), then the map $F_{f},f_{F}$ we construct from it is likewise continuous (respectively homomorphic).

 Lastly, note that we have for $t\in S(u)$,
    \begin{align*}
        f_{F}(t,\chi)
        =
        F\big(\mfs(t), (\mfs^*)\inv(\chi)\big)
        \overset{\eqref{eq:F:tau}}{=}
        (\mfs^*)\inv(\chi) \bigl(\mfs(t)\bigr)
        =
        (\mfs^*\circ (\mfs^*)\inv)(\chi)(t)
        =
        \chi(t)
    \end{align*}
    and for $\tau=(z,t)\in E_{S}(u)$,
    \begin{align*}
        F_{f}(\tau,\kappa)
        =
        z\, f\big( t,\mfs^*(\kappa)\big)
        \overset{\eqref{eq:f:t}}{=}
        z\, \mfs^*(\kappa)(t)
        =
        z\, \kappa(\mfs(t))
        =
        \kappa(z\cdot (1,t))
        = 
        \kappa(\tau).
    \end{align*}
    This proves that $f_{F}$ satisfies~\eqref{eq:f:t} if $F$ satisfies~\eqref{eq:F:tau}, and that $F_{f}$ satisfies~\eqref{eq:F:tau} if $f$ satisfies~\eqref{eq:f:t}. An application of Proposition~\ref{prop:trivial Weyl twist} now proves the claim regarding the Weyl twist.
    
    For the claim about $\cs$-algebras, note that
    \begin{align*}
        \csr(G) = \csr (G;\E)
        &\cong
        \csr
        \bigl(
            \Weylgpd  ;
            \Weyltwist 
        \bigr)
        &&\text{by Corollary~\ref{cor:Cartan pair iso}}
        \\
        &\cong 
        \csr
        \bigl(
            \Weylgpd
        \bigr)
        &&\text{by the above}
        \\
        &\cong 
        \csr
        \bigl(
            \Weylgpd[\widehat{S}]
        \bigr)
        &&\text{by the isomorphism at \eqref{eq:mfs:iso of gpds}}.
    \end{align*}
\end{proof}

\section{$\cs$-diagonals}\label{sec:diag}

In this section, we will use the previous results to identify $\cs$-diagonals coming from open subgroupoids. 
\begin{notation}
    If $\pi\colon \E\to G$ is  a twist and $u$ a unit, then for two elements $e,\sigma\in \E(u)$, we let $[e,\sigma]\coloneqq  e\inv\sigma\inv e\sigma$ denote the commutator.\footnote{Our convention for $[e,\sigma]$ differs from that in \cite[Section 7]{Renault:2023:Ext}.} 
    If $S\leq G$ is a subgroupoid, we write
    \[
        [e, \E_{S} ]
        \coloneqq  
        \left\{
        [e,\sigma]
        :
        \sigma\in  \E_{S} (u)
        \right\}
     \quad\text{ and }\quad
        [\pi(e),S]
        \coloneqq  
        \left\{
        [\pi(e),s]
        :
        s\in S(u)
        \right\}.
    \]
\end{notation}

\begin{thm}\label{thm:diag}
    Suppose 
  $\E$ is a twist 
  over a \LCH, \etale\
  groupoid~$G$, and~$S$ is an open subgroupoid of~$G$, so that $B\coloneqq  \csr(S; \E_{S} )$ can be regarded as a subalgebra of $A\coloneqq  \csr(G;\E)$. Then the following statements  are equivalent.
  \begin{enumerate}[label=\textup{(\roman*)}]
    
    \item\label{it:diag:B,1} 
    $B$ is a $\cs$-diagonal in $A$ for which
    \begin{equation}\label{eq:thm:bisection}
        \{
        h\in C_{c}(G;\E):
        \suppoG(h) \text{ is a bisection}
        \} \subset N(B).
    \end{equation}
    
    \item\label{it:diag:B,2}
    $B$ is a $\cs$-diagonal in $A$ for which the set
    \begin{equation*}
       \{
        a\in A:
        \suppoG(j(a)) \text{ is a bisection}
        \} 
        \cap  N(B)
    \end{equation*}
    generates $A$.
    
    \item\label{it:diag:S}
    The restricted twist $\E_{S}$ is abelian, closed, and normal in $\E$, and
      \begin{align}
        \label{eq:diag:S}
        \left\{e\in \Iso{\E}: 
            \pi\text{ is injective on }[e, \E_{S} ]\right\} \subset \E_{S}.
    \end{align}   
  \end{enumerate}
  If the above equivalent conditions are satisfied, then \eqref{eq:diag:S} is in fact an equality of sets, and~$S$ is maximal among all subgroupoids of $\Iso{G}$ whose restricted twist is abelian.
\end{thm}

 We will start with a series of lemmas. 
 It is a consequence of \cite[Proposition 3.14]{DWZ:2025:Twist} that in all three statements of Theorem~\ref{thm:diag}, we have that $  S\z$ must be all of $\cU =G\z$. In particular, we are  in the setting of Assumption~\ref{as:E,G,S}, and so we may resume the notation from the previous sections:  $\E/ \E_{S} =G/S$ is the quotient groupoid whose elements we write as $\dot{e}$ for $e\in \E$; $\twPD=\bigsqcup_{u\in\cU }\twPD(u)$ is the twisted dual (Definition~\ref{def:twPD}), which is homeomorphic to $\widehat{B}$ (Proposition~\ref{prop:spec B}); and $\HleftX$ denotes the action of $\E$ and of $G/S$ on $\twPD$ (Proposition~\ref{prop:E:HleftX is an action} resp.\ Corollary~\ref{cor:G/S:HleftX is an action}).

\begin{lemma}\label{lem:iso of Q x twPD}
   In the setting of Assumption~\ref{as:E,G,S}, let 
   $(\dot{e},\kappa)\in \Weylgpd $. Then $(\dot{e},\kappa)$ is an isotropy point if and only if $r(e)=s(e)$ and 
    \begin{align}
        \label{eq:condition on kappa,v2}
        \kappa(f)=1\quad\text{ for every  }f\in [e, \E_{S} ].
    \end{align}
\end{lemma}

\begin{proof}
    By definition, 
   $(\dot{e},\kappa)$
    is an isotropy point if and only if its range $\dot{e}\HleftX \kappa$ equals its source $\kappa$; in particular, $u\coloneqq   P(\dot{e}\HleftX \kappa)=r(e)$ must equal $P(\kappa)=s(e)$, and so $\dot{e}\in\Iso{(G/S)}$.  We have
    \begin{align}
        \kappa=\dot{e}\HleftX \kappa
        \iff 
        \label{eq:condition on kappa}
        \kappa(\tau)=\kappa(e\inv \tau e)
       \quad \text{ for every }\tau\in  \E_{S} (u).
    \end{align}
    Since $\kappa$ is a homomorphism, we have
    \[
    \kappa\bigl( [e,\tau]\bigr)=\kappa(e\inv \tau\inv e \tau)=\kappa(e\inv \tau e)\inv \kappa(\tau).\qedhere
    \]
\end{proof}

\begin{lemma}\label{lem:not emptyset with Ad => not principal}
In the setting of Assumption~\ref{as:E,G,S}, if $e\in \Iso{\E}$ is such that $\pi$ is injective on $[e, \E_{S} ]$,
    then there exists $\kappa\in \twPD $ such that 
   $(\dot{e},\kappa) \in \Iso{\Weylgpd }$.
\end{lemma}

\begin{proof}
    Let $u\coloneqq  r(e)=s(e)$. Note that $[e, \E_{S} ]$ is a subgroup of $ \E_{S} (u)$. Indeed, if $\sigma,\tau\in \E_{S} (u)$, then
    \begin{align*}
            [ e ,\sigma]\inv [ e ,\tau]
            =
            \left( e\inv \sigma\inv  e  \sigma\right)\inv
             e\inv \tau\inv  e  \tau
            =
            \sigma\inv  e\inv \sigma  e  
             e\inv \tau\inv  e  \tau
            =
            \sigma\inv ( e\inv \sigma \tau\inv  e ) \tau.
    \end{align*}
    Since $ \E_{S} (u)$ is a normal and abelian subgroup of $\E(u)$, we have that $ e\inv \sigma \tau\inv  e $ commutes with $\sigma\inv$ and that $\sigma \tau\inv= (\sigma\inv \tau)\inv$, so that
        \begin{align*}
            [ e ,\sigma]\inv [ e ,\tau]
            =
             e\inv (\sigma\inv \tau)\inv  e  (\sigma\inv \tau)
            =
            [ e ,\sigma\inv \tau]
            \in [ e , \E_{S} ]
        \end{align*}
    as needed.

    Since $S(u)$ is discrete, any set-theoretic section $S(u)\to \E_{S}(u)$ of $\pi$ that fixes $u$ gives rise to a normalized, symmetric $2$-cocycle on $S(u)$ (Lemma~\ref{lem:c_s}). Since $S(u)$ is furthermore abelian, an old result by \citeauthor{Kleppner:1965:Multipliers} \cite[Lemma 7.2]{Kleppner:1965:Multipliers} tells us that this $2$-cocycle is cohomologous to a 2-coboundary. Since the twist induced by this $2$-cocycle is isomorphic to $\E_{S}(u)$ (Lemma~\ref{lem:c_s}), the twist is trivializable, i.e., 
    there exists a $\mathbb{T}$-equivariant homomorphism $\Phi\colon\mathbb{T}\times S(u) \to \E_{S}(u)$ such that $\pi(\Phi(z,s))=s$. Our original assumption that $\pi$ be injective on $[e,\E_{S}]$ thus implies that, if we let $ \mathcal{F}\coloneqq \Phi\inv([e,\E_{S}])$ and $\mathrm{pr}_{S}\colon\mathbb{T}\times S(u)\to S(u)$ the projection map, then there exists a unique map
    \[\rho_0\colon \mathrm{pr}_{S}( \mathcal{F})\to \mathbb{T}
    \quad\text{ such that }\quad
     \mathcal{F}=\{(\rho_0(s),s):s\in\mathrm{pr}_{S}( \mathcal{F})\}.\]
    Since $ \mathcal{F}$ is a subgroup, we must have $(\rho_0(s),s)(\rho_0(t),t)=(\rho_0(st),st)$ for all $s,t\in\mathrm{pr}_{S}( \mathcal{F})$, so that $\rho_{0}$ is a homomorphism. We claim that we can extend $\rho_0$ to a homomorphism on all of $S(u)$.
    To this end, consider the (clearly non-empty) set 
    \begin{align*}
            P \coloneqq  {}
            \bigl\{
            (\rho,H) 
            :{}
            &\tilde\pi( \mathcal{F})\leq H\leq S(u) \text{ is a subgroup};    
            \bigr.
            \\
            \bigl.
            &
            \rho\colon 
            H
            \to
            \mathbb{T}
            \text{ is a homomorphism such that }
            \rho|_{\mathrm{pr}_{S}( \mathcal{F})}=\rho_0
            \bigr\}
        \end{align*}
    with the partial order given by
        \[
            (\rho,H) \leq  (\rho',H')
            :\iff
            \left( H\leq H' \quad\text{ and }\quad \rho= \rho'|_{H}\right).
        \]
     Clearly, any chain $\mathcal{K}=\{(\rho_{i},H_{i})\}_{i\in I}$ in $P$ has an upper bound in $P$, namely by defining on the supergroup $\cup_{i\in I} H_{i}$ of $\tilde\pi( \mathcal{F})$ the homomorphism $H_i\ni h \mapsto \rho_{i}(h)$.
     By Zorn's Lemma, there must exist a maximal element $(\rho, H)$ in $P$. To show that $H=S(u)$, assume otherwise, meaning that there exists $s\in S(u)\setminus H$, so that $\tilde{H}\coloneqq  \langle H, s\rangle$ is a proper supergroup of $H$. There are two cases:
\begin{description}
            \item[Case 1] If there exists a minimal $m\in\mathbb{Z}_{>1}$ such that $s^m\in H$, then let $z_0$ be any $m$th root of $\rho(s^{m})\in\mathbb{T}$. We let
            \begin{align*}
                \tilde{\rho}\colon \tilde{H}\to\mathbb{T}
                \quad
                \text{ be given by }
                \quad
                \tilde\rho(hs^k)= \rho(h)z_0^k
                \quad\text{ where }\quad h\in H, k\in\mathbb{Z}.
            \end{align*}
            This map is defined on all of $\tilde{H}$ since $S(u)$ is abelian, and it is well defined: if $hs^k=h's^{k'}$, then $s^{k-k'}=h\inv h' \in H$, so that $k-k'=m\ell$ for some $\ell\in\mathbb{Z}$ and hence
            \[
                \rho(h')z_0^{k'} [\rho(h)z_0^k]\inv
                =
                \rho(h\inv h')z_0^{k'-k}
                =
                \rho(s^{k-k'})z_0^{-m\ell}
                =
                (\rho(s^{m})z_0^{-m})^{\ell}.
            \]
            This equals $1$ 
            by choice of $z_0$, proving that $\rho(h')z_0^{k'} =  \rho(h)z_0^k$. The map is furthermore a homomorphism that extends the homomorphism $\rho$.

            \item[Case 2] If no $m\in\mathbb{Z}\setminus\{0\}$ satisfies $s^m\in H$, then we may let
            \begin{align*}
                \tilde{\rho}\colon \tilde{H}\to\mathbb{T}
                \quad
                \text{ be given by }
                \quad\tilde\rho(hs^k)= \rho(h)
                \quad\text{ where }\quad h\in H, k\in\mathbb{Z}.
            \end{align*}
            As in the first case, this map is defined on all of $\tilde{H}$, it is well defined since $hs^k=h's^{k'}$ implies first $k-k'=0$ and thus $h=h'$, and it is a homomorphism that extends $\rho$.
        \end{description}
     In either of the two cases, the existence of $(\tilde\rho,\tilde H)$ contradicts maximality of $(\rho,H)$. This proves that $H=S(u)$, and so $\rho$ is an element of $\widehat{S(u)}$, the (untwisted) dual group. 
     
     Now, since $\Phi$ is an equivariant isomorphism, the map
    \[
        \kappa\colon \E_{S}(u)\to\mathbb{T}
        \quad\text{ determined by }\quad
        \kappa (\Phi(z,s))= z\, \overline{\rho (s)}
    \]
    is an element of $\twPD(u)$. Furthermore, if $\sigma\in [e,\E_{S}(u)]$, then since $\rho$ extends $\rho_0$, we have $\sigma=\Phi(\rho(s),s)$  for $s\coloneqq  \pi(\sigma)$, so that
    \[
        \kappa (\sigma)
        =
        \kappa \bigl(\Phi(\rho(s),s)\bigr)
        =
        \rho(s)\, \overline{\rho (s)}
        =
        1,
    \]
    meaning that $\kappa$ satisfies Equation~\eqref{eq:condition on kappa,v2}. By Lemma~\ref{lem:iso of Q x twPD}, 
   $(\dot{e},\kappa)$ is an isotropy point of $\Weylgpd $.
\end{proof}

\begin{proof}[Proof of Theorem~\ref{thm:diag}]
Since every $\cs$-diagonal is, in particular, a Cartan subalgebra \cite[Definition 2.11]{pitts2022normalizers}, a lot of the heavy lifting can be done by invoking \cite[Theorem 3.3]{DWZ:2025:Twist}:  firstly, we automatically have that \ref{it:diag:B,1}$\iff$\ref{it:diag:B,2}. Then, note that the following statements are equivalent:
\begin{itemize}
    \item $B=\csr(S; \E_{S} )$ is  a $\cs$-diagonal in $A=\csr(G;\E)$;
    \item $(A,B)$ is a Cartan pair whose Weyl groupoid $W_{B\subset A}$ is principal \cite{Kum:Diags};
    \item $(A,B)$ is a Cartan pair, and the groupoid $\Weylgpd $ is principal (Theorem~\ref{thm:Weyl groupoid}).
\end{itemize}
    Thus, by\footnote{Here, we have made use of the fact that the maps
    \begin{align*}
    [e,\E_{S}]
    &\to \left\{ \sigma\inv e\sigma : \sigma\in  \E_{S} (u)\right\},
    &
    f&\mapsto ef,
    \qquad\text{ and}
    \\
    [\pi(e),S]
    &\to
    \left\{ s\inv \pi(e)s : \sigma\in  S (u)\right\},
    &
    \pi(f)&\mapsto \pi(ef),
    \end{align*}
    are bijections, so that the conditions~\eqref{eq:thm:DWZ:maximal} respectively \eqref{eq:thm:DWZ:ricc} coincide with the statements (\textsf{max}) and (\textsf{ricc}) as they appear in \cite[Theorem 3.3(iii)]{DWZ:2025:Twist}.} \cite[Theorem 3.3, (i)$\iff$(iii)]{DWZ:2025:Twist}, we conclude that the statement in Theorem~\ref{thm:diag}\ref{it:diag:B,1}  is equivalent to the following:
  \begin{enumerate}[label=\textup{(\roman*)}, start=5]
  \em
    \item\label{it:thm:DWZversion} 
   ~$S$ satisfies \eqref{eq:thm:DWZ:maximal} and \eqref{eq:thm:DWZ:ricc} of Theorem~\ref{thm:DWZ},
     ~$S$ is closed and normal in~$G$, and the groupoid $\Weylgpd $ is principal.
  \end{enumerate}
\noindent Thus, it suffices to show that 
\ref{it:thm:DWZversion}$\iff$\ref{it:diag:S}. Since~$S$ is closed (resp.\ normal) in~$G$ if and only if $\E_{S}$ is closed (resp.\ normal) in $\E$, we see that in all three statements, we do not need to worry about those properties anymore and may simply assume them.

    \ref{it:thm:DWZversion}$\implies$\ref{it:diag:S}: Thanks to~\eqref{eq:thm:DWZ:maximal}, we can invoke  \cite[Corollary 3.17]{DWZ:2025:Twist} to conclude that $ \E_{S} $ is abelian and contains $G\z$. Since we further assumed~$G$ to be \etale\ and~$S$ to be open, this shows that we are in the setting of Assumption~\ref{as:E,G,S}. 
It only remains to check that \eqref{eq:diag:S} holds, so let $e\in\Iso{\E}$ be such that $\pi$ is injective on $[e, \E_{S} ]$; we must show that $e\in\E_{S}$. By Lemma~\ref{lem:not emptyset with Ad => not principal}, there exists $\kappa\in \twPD $ such that 
   $(\dot{e},\kappa)\in\Iso{\Weylgpd }$. Since $\Weylgpd $ is principal, this implies $(\dot{e},\kappa)$ is a unit, i.e., that $e\in  \E_{S} $.

     \ref{it:diag:S}$\implies$\ref{it:thm:DWZversion}: 
     Since $\pi$ is injective on the empty set, we have
     \begin{equation*}
        \{e\in \Iso{\E}: |[e,\E_{S}]|=0\}
        \overset{\eqref{eq:diag:S}}{\subset }
        \E_{S}
        .
     \end{equation*}
     As $|[e,\E_{S}]|=0$ if and only if $r(e)\notin S$, this implies that
     \(
        G\z \subset \E_{S}.
     \)
     Similarly, since $\pi$ is injective on any singleton-set, we have
     \[
        \{e\in \Iso{\E}: |[e,\E_{S}]|=1\}
        \overset{\eqref{eq:diag:S}}{\subset }
        \E_{S}.
     \]
     Since $\E_{S}$ is abelian, we also have the other containment, so we arrive at:
     \begin{equation}\label{eq:1}
        G\z\subset \E_{S}=\{e\in \Iso{\E}: 1=|[e,\E_{S}]|\}.
    \end{equation}
     
    Since $\E_{S}$ is open, this shows that \eqref{eq:thm:DWZ:maximal} holds. Now, if $e\in\Iso{\E}$ is such that $\left|[e,\E_{S}]\right|=\left|[\pi(e),S\vphantom{\E_{S}}
      ]\right|<\infty$, then $\pi$ is injective on $[e,\E_{S}]$, so Assumption~\eqref{eq:diag:S} combined with \eqref{eq:1} implies that $1=\left|[e,\E_{S}]\right|$. In particular, 
    \[
    \left\{
      e\in \Iso{\E}: 1<\left|[e,\E_{S}]\right|=\left|[\pi(e),S\vphantom{\E_{S}}
      ]\right|<\infty
      \right\}
      =
      \emptyset,
    \]
    which implies that \eqref{eq:thm:DWZ:ricc} holds.  
    
    It  remains to show that $\Weylgpd $ is principal, so let $(\dot{e},\kappa)$ be an isotropy point. Take $f_{1},f_{2}\in [e, \E_{S} ]$ such that $\pi(f_{1})=\pi(f_{2})$, so there exists $z\in\mathbb{T}$ such that $f_{1}=z\cdot f_{2}$. Since $(\dot{e},\kappa)$ is an isotropy point, it follows from Lemma~\ref{lem:iso of Q x twPD} that $u\coloneqq  r(e)$ equals $s(e)$ and that Equation~\eqref{eq:condition on kappa,v2} holds; in particular, $\kappa(f_{1})=1=\kappa(f_{2})$. Since $\kappa\in \twPD(u)$, so that $\kappa(z\cdot f_{2})=z\kappa(f_{2})$, we deduce that $z=1$ and so $f_{1}=f_{2}$. In other words, $\pi$ is injective on the set $[e, \E_{S} ]$. By Assumption~\eqref{eq:diag:S}, this implies $e\in \E_{S}$, so $(\dot{e},\kappa)$ is a unit in $\Weylgpd $. This finishes the proof of \ref{it:diag:S}$\implies$\ref{it:thm:DWZversion}.

Regarding the final remark in the theorem,  ~$S$ is maximal among all subgroupoids of $\Iso{G}$ whose restricted twist is abelian because of \eqref{eq:1}, and 
\[
    \E_{S}
    \overset{\eqref{eq:1}}{=}
    \left\{e\in \Iso{\E}: 
    |[e, \E_{S} ]|=1\right\}
    \subset
    \left\{e\in \Iso{\E}: 
    \pi\text{ is injective on }[e, \E_{S} ]\right\}
    \overset{\eqref{eq:diag:S}}{\subset} \E_{S},
\]
so we must have equality in \eqref{eq:diag:S}.
\end{proof}

\begin{proof}[Proof of Theorem~\ref{thm:B}]
    Since  Theorem~\ref{thm:B} assumes that $S$ is normal, the equivalent conditions of Theorem~\ref{thm:diag} become: $B$ is a $\cs$-diagonal if and only if $\E_{S}$ is abelian and closed and contains all $\{e\in\Iso{\E}:\pi\text{ is injective on }[e,\E_{S}]\}$. This containment means exactly that, if $e\in \E(u)\setminus\E_{S}(u)$, then  there exists $\sigma\in \E_{S}(u)$ such that $e\inv \sigma\inv e\sigma \neq u$ but $\pi(e\inv \sigma\inv e \sigma)=u$. In other words, $e\sigma=z\cdot \sigma e $ for some $z\neq 1$.
\end{proof}

\subsection{Corollaries and applications of Theorem~\ref{thm:diag}}\label{ssec:corollaries:B}

\begin{corollary}\label{cor:partial answer}
    Suppose 
  $\E$ is a twist 
  over a \LCH, \etale\
  groupoid~$G$, and that~$S$ is a clopen and normal subgroupoid of~$G$ for which $\E_{S}$ is abelian. If~$S$ satisfies Condition~\eqref{eq:diag:S},
    then the action of $G/S$ on $\twPD$ defined in Corollary~\ref{cor:G/S:HleftX is an action} is free.
\end{corollary}
\begin{proof}
    By Theorem~\ref{thm:diag}, \ref{it:diag:S}$\implies$\ref{it:diag:B,1}, $\csr(S; \E_{S} )$ is a $\cs$-diagonal in $\csr(G;\E)$.  By \cite{Kum:Diags}, the associated Weyl groupoid is principal which, by Theorem~\ref{thm:Weyl groupoid}, is given by $\Weylgpd $. The claim now follows from Remark~\ref{rmk:free=principal etc}.
\end{proof}

Corollary~\ref{cor:partial answer} in particular implies the following discrete case of \cite[Theorem 7.12]{Renault:2023:Ext}, where the twist was furthermore assumed to be irreducible.
\begin{corollary}
    Suppose that $E$ is a twist over a discrete abelian group $\Gamma$, and that $ S \leq\Gamma$ is such that $ E_{ S }$ is maximal abelian in $ E $. Then 
    the action of $\Gamma/ S $ on $\twPD[ E ]$ is 
    free. 
\end{corollary}

\begin{proof}
    Since $\Gamma$ is discrete abelian, $S$ is clopen and normal, and since $E_{S}$ is maximal abelian,~\eqref{eq:diag:S} holds.
    We may thus invoke Corollary~\ref{cor:partial answer} to deduce that the action of $\Gamma/S$ on $\twPD[E]$ is free.
\end{proof}

\begin{corollary}
  Suppose~$G$ is a \LCH, \etale\
  groupoid. Then~$G$ is principal if and only if $C_0(G\z)$ is a $\cs$-diagonal in $\csr(G;\E)$ for any (and hence every) twist $\E$ over~$G$.
\end{corollary}

\begin{proof}
    Suppose~$G$ is principal, and $\E$ is any twist over~$G$. Then $S\coloneqq  \Iso{G}=G\z$ is clopen and normal in~$G$, and $ \E_{S} =\pi\inv(G\z)=\iota(\mathbb{T}\times G\z)$ is abelian. Since
    \[
        \Iso{\E}\setminus \E_{S}  = \pi\inv (\Iso{G}\setminus S) = \pi\inv (\emptyset)=\emptyset,
    \]
    Condition~\eqref{eq:diag:S} is vacuously true. Theorem~\ref{thm:diag}, \ref{it:diag:S}$\implies$\ref{it:diag:B,1}, shows that $C_0(G\z)$ is a $\cs$-diagonal in $\csr(G;\E)$.

    Conversely, suppose $B=C_0(G\z)$ is a $\cs$-diagonal in $\csr(G;\E)$ for some given twist $\E$. Since $G\z$ is normal in~$G$, it follows from \cite[Corollary 3.23]{DWZ:2025:Twist} that the conditions in Theorem~\ref{thm:diag}\ref{it:diag:B,1} are satisfied for $B$ when thought of as $\csr(S; \E_{S} )$ for $S=G\z$. By \ref{it:diag:B,1}$\implies$\ref{it:diag:S} of Theorem~\ref{thm:diag}, we know that Condition~\eqref{eq:diag:S} holds.
    As $ \E_{S} =\pi\inv(G\z)=\iota(\mathbb{T}\times G\z)$, centrality  of $\iota$ implies $[e, \E_{S} ]=\{r(e)\}$ for any $e\in \Iso{\E}$; in particular, $\pi$ is injective on $[e,\E_{S}]$. Thus, Condition~\eqref{eq:diag:S} becomes the assertion that the set
    \(
        \Iso{\E}
    \)
    coincides with
    \(  \E_{S},
    \)
    which implies that $\Iso{G}=\pi(\Iso{\E})$ equals $\pi(\E_{S})=G\z$, meaning that~$G$ is principal.
\end{proof}

We record a remarkable corollary for the untwisted case:

\begin{corollary}\label{cor:diag:untwisted}
    Suppose~$G$ is a \LCH, \etale\ 
    groupoid, and~$S$ is an open subgroupoid of~$G$. Then the following statements  are equivalent.
  \begin{enumerate}[label=\textup{(\roman*)}]
    \item\label{it:1}
    $\csr(S)$ is a $\cs$-diagonal in $\csr(G)$ for which
    \begin{equation*}  
        \{
        h\in C_{c}(G):
        \suppo(h) \text{ is a bisection}
        \} \subset N(\csr(S)).
    \end{equation*}
    \item\label{it:2}
    $\csr(S)$ is a $\cs$-diagonal in $\csr(G)$, and~$S$ is normal in~$G$;
    \item\label{it:3} $S=\Iso{G}$ is abelian.
  \end{enumerate}
\end{corollary}

\begin{proof}
The implication \ref{it:2}$\implies$\ref{it:1} follows from \cite[Corollary 3.23]{DWZ:2025:Twist}, and \ref{it:1}$\implies$\ref{it:2} is contained in the proof of \cite[Theorem 3.3, (ii)$\implies$(iv)]{DWZ:2025:Twist}.

\ref{it:1}$\implies$\ref{it:3}:
By Theorem~\ref{thm:diag}, \ref{it:diag:B,1}$\implies$\ref{it:diag:S}, $\E_{S}$ and hence~$S$ is abelian, and Condition~\eqref{eq:diag:S} holds.  If $e=(z,g)\in \E(u)$, then 
\[
    [e, \E_{S} ]
    =
    \{ (1,g\inv s\inv g s) : s\in S(u)\},
\]
so that $\pi\colon (1,g\inv s\inv g s)\mapsto g\inv s\inv g s$ is injective on $[e, \E_{S} ]$ for every $e\in \Iso{\E}$. In other words, Condition~\eqref{eq:diag:S} implies $\E_{S}=\Iso{\E}$, i.e., $S=\Iso{G}$.

\ref{it:3}$\implies$\ref{it:2}: By assumption on~$S$, we have that $S=\Iso{G}$ is open and abelian. Since $\Iso{G}$ is normal and closed in~$G$, it follows that $\mathbb{T}\times S$ is abelian, closed, and normal in $\mathbb{T}\times \E$. Since Condition~\eqref{eq:diag:S}  holds vacuously for $\mathbb{T}\times S$, we have verified all conditions in Theorem~\ref{thm:diag}\ref{it:diag:S} applied to the situation in which $\E$ is the trivial twist over~$G$, and so by
Statement~\ref{it:diag:B,1} of that same theorem,
$\csr(S)\cong \csr(S;\mathbb{T}\times S)$ is a $\cs$-diagonal in $\csr(G)\cong \csr(G;\mathbb{T}\times G)$.
\end{proof}

\begin{example}[continuation of Examples~\ref{ex:c_theta:setup} and~\ref{ex:c_theta:Weyl groupoid}]\label{ex:c_theta:diag}
   We go back to the rotation algebras $A_{\theta} \cong \csr(\mathbb{Z}^2,\cocycle_{\theta})$ with $\cocycle_{\theta}$ as at~\eqref{eq:c_theta} and $\E_{\theta}=\mathbb{T}\times_{\cocycle_{\theta}}\mathbb{Z}^2$ the associated twist; let $\lambda=\mathsf{e}^{2\pi i\theta}$. We will verify that all Cartan subalgebras arising from subgroups~$S$ of $\mathbb{Z}^{2}$ are, in fact, $\cs$-diagonals.
   By Theorem~\ref{thm:diag}, \ref{it:diag:S}$\implies$\ref{it:diag:B,1},  it suffices to check that Condition~\eqref{eq:diag:S} is satisfied. 
    Since $\mathbb{Z}^{2}$ is abelian, we have $\pi([f, \E_{\theta,S} ])=\{(1;0,0)\}$ for any $f\in \E$; so 
    we must show that, whenever $[f,\E_{\theta,S}]$ is a singleton, then $\pi(f)\in S$. If we fix $f=(z;\vec{h})$,  then by Equation~\eqref{eq:c_theta:conjugation},
    \begin{align}\label{eq:theta:commutators}
        f\inv(w;\vec{s})\inv f(w;\vec{s}) 
        =
        \bigl(
        \lambda^{s_{1} h_{2} - s_{2}h_{1}};0,0\bigr).
    \end{align}
    We now have to make a case distinction:

    \myparagraph{Case 1: $\theta\notin\mathbb{Q}$}
    As before, we only need to consider $S=\mathbb{Z}\cdot (m,n)$ with $\gcd(m,n)=1$. In this case, Equation~\eqref{eq:theta:commutators} shows that 
    \[
    [f,\E_{\theta,S}]
    =
    \left\{
    \bigl(
        \lambda^{r( m h_{2} - nh_{1})};0,0\bigr)
    :
    r\in\mathbb{Z}
    \right\},
    \]
    so since $\theta$ is irrational, we conclude that  $ [f,\E_{\theta,S}]$ is a singleton
    if and only if $m h_{2} = nh_{1}$.  Since $\gcd(m,n)=1$, this implies that $h_{1}=mk$ for some $k\in\mathbb{Z}$, so $h_{2}=nh_{1}/m=nk$, and thus 
    $\pi(f)=\vec{h}=(mk,nk)$ is an element of $\mathbb{Z}\cdot (m,n)=S$. We conclude that $\csr(\mathbb{Z}\cdot (m,n),\cocycle_{\theta})$ is a $\cs$-diagonal in $\csr(\mathbb{Z}^{2},\cocycle_{\theta})$.

    \myparagraph{Case 1: $\theta\in\mathbb{Q}$} As always, the case $\theta=0$ is trivial, so let $\theta=p/q$ for  $p\in\mathbb{Z}^{\times}$, $q\in \mathbb{N}$, and $\gcd(p,q)=1$; fix $S=\mathbb{Z}\cdot (k,0)\oplus \mathbb{Z}\cdot (m,n)$ with $nk=q$. In this case, Equation~\eqref{eq:theta:commutators} shows that 
    \[
    [f,\E_{\theta,S}]
    =
    \left\{
    \bigl(
        \lambda^{lkh_{2}+r(mh_{2}  - nh_{1})};0,0\bigr)
    :
    l,r\in\mathbb{Z}
    \right\}.
    \]    
    Since $\gcd(p,q)=1$ and $q=nk$, we deduce that $[f,\E_{\theta,S}]$ is a singleton if and only if $\pi(f)=\vec{h}$ satisfies
    \[
     l  kh_{2} + r(mh_{2}-nh_{1}) \in nk\mathbb{Z}
    \text{ for all }  l ,r\in\mathbb{Z}.
    \]
    Choosing $r=0,l=1$ implies that $h_{2}=nc$ for some $c\in\mathbb{Z}$. The above condition then becomes for $l=0,r=1$:
    \[
     n(mc-h_{1}) \in nk\mathbb{Z},\quad\text{ i.e., }\quad
     h_{1} = mc+kd \text{ for some }d\in\mathbb{Z}.
    \]
    We have shown that $\vec{h}=d(k,0)+c(m,n)$, which is an element of $ S$, so that $\csr(\mathbb{Z}\cdot (k,0)\oplus \mathbb{Z}\cdot (m,n),\cocycle_{p/q})$ is a $\cs$-diagonal in $\csr(\mathbb{Z}^{2},\cocycle_{p/q})$.
\end{example}

\section{Open problems}\label{sec:questions}

\begin{problem}
    Prove Theorem~\ref{thm:diag} directly (for example along the same lines of  \cite[Theorem 3.3]{DWZ:2025:Twist}) without appealing to Theorem~\ref{thm:A} and \cite{Kum:Diags}.
\end{problem}

Corollary~\ref{cor:partial answer} and Remark~\ref{rmk:effective} are partial answers to the following problem, posed by Jean Renault:
\begin{problem}\label{prob:Jean}
        Given a twist $\E$ over a locally compact (not necessarily Hausdorff or \etale) groupoid~$G$, find a necessary and
        sufficient condition on a closed normal subgroupoid~$S$ of~$G$ so that the action of
        $G/S$ on the twisted spectrum $\twPD$ is free or topologically free.
\end{problem}

\printbibliography

\end{document}